\documentclass{amsart}

\usepackage{pinlabel}
\usepackage{amsmath,amsfonts,amssymb,latexsym,mathrsfs,stmaryrd}
\usepackage{color}
\usepackage{graphicx}
\usepackage{graphics}
\usepackage{enumerate}
\usepackage{amscd} 
\usepackage{amsxtra}
\usepackage{epic,eepic}
\usepackage{hyperref} 

\newtheorem{theorem}{Theorem}[section]
\newtheorem{prop}[theorem]{Proposition}
\newtheorem{proposition}[theorem]{Proposition}

\newtheorem{thm}[theorem]{Theorem}
\newtheorem{lemma}[theorem]{Lemma}

\newtheorem{assumption}[theorem]{Assumption}

\theoremstyle{definition} 
\newtheorem{defn}[theorem]{Definition}

\newtheorem{rmk}[theorem]{Remark}

\newcommand{\cE}{\mathcal{E}}

\newcommand{\cU}{\mathcal{U}}
\newcommand{\cX}{\mathcal{X}}

\newcommand{\C}{\bb{C}}

\newcommand{\qu}{/\kern-.7ex/}
\newcommand{\lqu}{\backslash \kern-.7ex \backslash}
\newcommand{\bb}[1]{\ensuremath{\mathbb{#1}}}

\title{Seidel elements and mirror transformations for toric stacks}

\author{Fenglong You} 
\address{Department of Mathematics\\ The Ohio State University\\
  100 Math Tower\\
  231 West 18th Ave.\\Columbus, OH 43210, USA}
\email{you.111@osu.edu}


\begin{document}
\begin{abstract} We give a precise relation between the mirror transformation and the Seidel elements for weak Fano toric Deligne-Mumford stacks. Our result generalizes the corresponding result for toric varieties proved by Gonz\'alez and Iritani in \cite{GI}.
\end{abstract}
\maketitle 

\begin{center}
\today
\end{center}
\tableofcontents

\section{Introduction}

In \cite{GI}, Gonz\'alez and Iritani gave a precise relation between the mirror map and the Seidel elements for a smooth projective weak Fano toric variety $X$. The goal of this paper is to generalize the main theorem of \cite{GI} to a smooth projective weak Fano toric Deligne-Mumford stack $\mathcal{X}$. 

Let $\mathcal{X}$ be a smooth projective weak Fano toric Deligne-Mumford stack, the mirror theorem can be stated as an equality between the $I$-function and the $J$-function via a change of coordinates, called mirror map (or mirror transformation). We refer to \cite {CCIT} and section 4.1 of \cite{Iritani} for further discussions. 

Let $Y$ be a monotone symplectic manifold. 
For a loop $\lambda$ 
in the group of Hamiltonian symplectomorphisms 
on $Y$, Seidel \cite{Seidel}  constructed an invertible element 
$S(\lambda)$ in (small) quantum cohomology 
counting sections of the associated Hamiltonian $Y$-bundle 
$E_{\lambda}\rightarrow \mathbb{P}^{1}$. 
The Seidel element $S(\lambda)$ 
defines an element in $Aut(QH(Y))$ 
via quantum multiplication and the map 
$\lambda\mapsto S(\lambda)$ gives a representation of 
$\pi_{1}(Ham(Y))$ on $QH(Y)$. 
The construction was extended to all symplectic manifolds by McDuff and Tolman in \cite{MT}. 
Let $D_{1},\ldots, D_{m}$ be the classes in $H^{2}(X)$ 
Poincar\'e dual to the toric divisors. 
When the loop $\lambda$ is a circle action, 
McDuff and Tolman \cite{MT} considered 
the Seidel element $\tilde{S}_{j}$ 
associated to an action $\lambda_{j}$ that fixes the toric divisor $D_{j}$. The definition of Seidel representation and Seidel element were extended to symplectic orbifolds by Tseng-Wang in \cite{TW}.

Given a circle action on $X$ (resp. $\mathcal{X}$), 
the Seidel element in \cite{GI} (resp. \cite{TW}) 
is defined using the small quantum cohomology ring. 
In this paper, we need to define it, for smooth projective Deligne-Mumford stack, 
with deformed quantum cohomology to include the bulk deformations. 
For weak Fano toric Deligne-Mumford stack, 
the mirror theorem in \cite{Iritani} shows that 
the mirror map $\tau(y)\in H^{\leq 2}_{orb}(\mathcal{X})$, 
therefore, we will only need bulk deformations with 
$\tau\in H^{\leq 2}_{orb}(\mathcal{X})$.

We consider the Seidel element 
$\tilde{S}_{j}$ associated to the toric divisor 
$D_{j}$ as well as the Seidel element $\tilde{S}_{m+j}$
corresponding to the box element $s_{j}$. 
The Seidel element in definition \ref{Seidel element} 
shows that $S=q_{0}\tilde{S}$ is 
a pull-back of a coefficient of the 
$J$-function $J_{\mathcal{E}_{j}}$ of 
the associated orbifiber bundle 
$\mathcal{E}_{j}$, hence we can use 
the mirror theorem for $\mathcal{E}_{j}$ 
to calculate $\tilde {S}_{j}$ when 
$\mathcal{E}_{j}$ is weak Fano. 

We extend the definition of the Batyrev element $\tilde{D}_{j}$ 
to weak Fano toric Deligne-Mumford stacks 
via partial derivatives of the mirror map 
$\tau(y)$. As analogues of 
the Seidel elements in B-model, 
the Batyrev elements can be explicitly 
computed from the $I$-function of 
$\mathcal{X}$. The following theorem states that 
the Seidel elements and the Batyrev elements 
only differ by a multiplication of a correction function. 
\begin{thm}
Let $X$ be a smooth projective toric Deligne-Mumford stack with $\rho^{S}\in cl\left(C^{S}(\mathcal{X})\right)$.
\begin{enumerate}
\item the Seidel element $\tilde{S}_{j}$ associated to the toric divisor $D_{j}$ is given by 
\[
\tilde{S}_{j}\left(\tau(y)\right)=exp\left(-g_{0}^{(j)}(y)\right)\tilde{D}_{j}(y)
\]
where $\tau(y)$ is the mirror map of $\cX$ and the function $g_{0}^{(j)}$ is given explicitly in (\ref{correction coefficient 1});
\item the Seidel element $\tilde{S}_{m+j}$ corresponding to the box element $s_{j}$ is given by 
\[
\tilde{S}_{m+j}\left(\tau(y)\right)=exp\left(-g_{0}^{(m+j)}\right)y^{-D_{m+j}^{S\vee}}\tilde{D}_{m+j}(y),
\]
where $\tau(y)$ is the mirror map of $\cX$ and the function $g_{0}^{(m+j)}$ is given explicitly in (\ref{correction coefficient 2}).
\end{enumerate}
\end{thm}

It appears that the correction coefficients in the above theorem coincide with the instanton corrections in theorem 1.4 in \cite{CCLT}. This phenomenon also indicates the deformed quantum cohomology of the toric Deligne-Mumford stack $\mathcal{X}$ is isomorphic to the Batyrev ring given in \cite{Iritani}.

\section*{Acknowledgments}
The author wants to thank professor Hsian-Hua Tseng for his guidance and lots of helpful discussions.

\section{Seidel Elements and $J$-functions}
\subsection{Generalities}

In this section, we will fix our notation and construct 
the Seidel elements of smooth projective Deligne-Mumford stacks using $\tau$-deformed quantum cohomology.

Let $\mathcal{X}$ be a smooth projective Deligne-Mumford stack, equipped with a $\mathbb{C}^{\times}$ action. 

\begin{defn}
The associated orbifiber bundle of the $\mathbb{C}^{\times}$-action is the $\mathcal{X}$-bundle over $\mathbb{P}^{1}$
\begin{equation*}
\mathcal{E}:=\mathcal{X}\times\left(\mathbb{C}^{2}\setminus \{0\}\right)/\mathbb{C}^{\times}\rightarrow \mathbb{P}^{1},
\end{equation*}
where $\C^\times$ acts on $\C^{2}\setminus\{0\}$ via the standard diagonal action.
\end{defn}

Let $\phi_{1},\ldots,\phi_{N}$ 
be a basis for the orbifold cohomology ring
$H^{*}_{orb}(\mathcal{X}):=H^{*}(\mathcal{IX};\mathbb{Q})$ 
of $\mathcal{X}$, where $\mathcal{IX}$ is the inertia stack of $\mathcal{X}$. 
Let $\phi^{1},\ldots,\phi^{N}$ 
be the dual basis of $\phi_{1},\ldots,\phi_{N}$ 
with respect to the orbifold Poincar\'e pairing. 
Furthermore, let $\hat{\phi}_{1},\ldots,\hat{\phi}_{M}$ 
denote a basis for the orbifold cohomology 
$H^{*}_{orb}(\mathcal{E}):=H^{*}(\mathcal{IE};\mathbb{Q})$ of $\mathcal{E}$.
Let $\hat{\phi}^{1},\ldots,\hat{\phi}^{M}$ 
be the dual basis of $\hat{\phi}_{1},\ldots,\hat{\phi}_{M}$
 with respect to the orbifold Poincar\'e pairing. 

We will use $X$ to denote the coarse moduli space of 
$\mathcal{X}$ and use $E$ to denote the coarse moduli space of $\mathcal{E}$. 
Then the $\mathbb{C}^{\times}$ action on $\mathcal{X}$ descends to the 
$\mathbb{C}^{\times}$ action on $X$ with $E$ 
being the associated bundle. Following \cite{McDuff} and \cite{GI}, 
there is a (non-canonical) splitting 
\begin{equation*}
H^{*}(\mathcal{E};\mathbb{Q})\cong H^{*}(E;\mathbb{Q})\cong H^{*}(X;\mathbb{Q})\otimes H^{*}(\mathbb{P}^{1};\mathbb{Q})\cong H^{*}(\mathcal{X};\mathbb{Q})\otimes H^{*}(\mathbb{P}^{1};\mathbb{Q}).
\end{equation*}

According to \cite{GI}, there is a unique 
$\mathbb{C}^{\times}$-fixed component 
$F_{\text{max}}\subset X^{\mathbb{C}^{\times}}$ 
such that the normal bundle of $F_{\text{max}}$ 
has only negative $\mathbb{C}^{\times}$-weights. 
Let $\sigma_{0}$ be the section associated to a fixed point in 
$F_{\text{max}}$. Following \cite{GI}, there is a splitting defined by this maximal section.
\begin{equation}
H_{2}(\mathcal{E};\mathbb{Z})/tors \cong H_{2}(E;\mathbb{Z})/tors \cong \mathbb{Z}[\sigma_{0}]\oplus (H_{2}(X,\mathbb{Z})/tors) \cong \mathbb{Z}[\sigma_{0}]\oplus (H_{2}(\mathcal{X},\mathbb{Z})/tors).
\end{equation}

Let $NE(X)\subset H_{2}(X;\mathbb{R})$ 
denote the Mori cone, i.e. the cone generated by effective curves and 
set 
\begin{equation*}
NE(X)_{\mathbb{Z}}:=NE(X)\cap (H_{2}(X,\mathbb{Z})/tors). 
\end{equation*}
Then, by lemma 2.2 of \cite{GI}, we have 
\begin{equation}
NE(E)_{\mathbb{Z}}=\mathbb{Z}_{\geq 0}[\sigma_{0}]+NE(X)_{\mathbb{Z}}. 
\end{equation}
Let $H_{2}^{sec}(E;\mathbb{Z})$ be the affine subspace of 
$H_{2}(E,\mathbb{Z})/tors$ which consists of the classes 
that project to the positive generator of 
$H_{2}(\mathbb{P}^{1};\mathbb{Z})$, 
we set 
\begin{equation*}
NE(E)_{\mathbb{Z}}^{sec}:=NE(E)_{\mathbb{Z}}\cap H_{2}^{sec}(E;\mathbb{Z}), 
\end{equation*}
then we obtain 
\begin{equation}
NE(E)_{\mathbb{Z}}^{sec}=[\sigma_{0}]+NE(X)_{\mathbb{Z}}. 
\end{equation}

We choose a nef integral basis $\{p_{1},\ldots,p_{r}\}$ of 
$H^{2}(\mathcal{X};\mathbb{Q})$, 
then there are unique lifts of $p_{1},\ldots,p_{r}$ in $H^{2}(\mathcal{E};\mathbb{Q})$ 
which vanish on $[\sigma_{0}]$. 
By abuse of notation, we also denote these lifts as 
$p_{1},\ldots,p_{r}$, these lifts are also nef. 
Let $p_{0}$ be the pullback of the positive generator of 
$H^{2}(\mathbb{P}^{1};\mathbb{Z})$ in $H^{2}(\mathcal{E};\mathbb{Q})$. 
Therefore, $\{p_{0},p_{1},\ldots,p_{r}\}$ is an integral basis of $H^{2}(\mathcal{E};\mathbb{Q})$.

 Let $q_{0},q_{1},\ldots,q_{r}$ be the Novikov variables of 
$\mathcal{E}$ dual to $p_{0},p_{1},\ldots,p_{r}$ and 
$q_{1},\ldots,q_{r}$ be the Novikov variables of 
$\mathcal{X}$ dual to $p_{1},\ldots,p_{r}$. 
We denote the Novikov ring of $\mathcal{X}$ 
and the Novikov ring of $\mathcal{E}$ by 
\begin{equation*}
\Lambda_{\mathcal{X}}:=\mathbb{Q}[[q_{1},\ldots,q_{r}]]\quad \text{and} \quad \Lambda_{\mathcal{E}}:=\mathbb{Q}[[q_{0},q_{1},\ldots,q_{r}]],
\end{equation*}
respectively. For each $d\in NE(X)_{\mathbb{Z}}$, we write 
\begin{equation*}
q^{d}:=q_{1}^{\langle p_{1},d\rangle}\cdots q_{r}^{\langle p_{r},d\rangle}\in \Lambda_{\mathcal{X}};
\end{equation*}
and for each $\beta\in NE(E)_{\mathbb{Z}}$, we write 
\begin{equation*}
q^{\beta}:=q_{0}^{\langle p_{0},\beta\rangle}q_{1}^{\langle p_{1},\beta\rangle}\cdots q_{r}^{\langle p_{r},\beta\rangle}\in \Lambda_{\mathcal{E}}.
\end{equation*}

The $\tau$-deformed orbifold quantum product is defined as follows:
\begin{equation}
\alpha\bullet_{\tau}\beta=\sum\limits_{d\in NE(X)_{\mathbb{Z}}}\sum\limits_{l\geq0}\sum\limits_{k=1}^{N}\frac{1}{l!}\langle\alpha,\beta,\tau,\ldots,\tau,\phi_{k}\rangle^{\mathcal{X}}_{0,l+3,d}q^{d}\phi^{k},
\end{equation}
\noindent the associated quantum cohomology ring is denoted by 
\begin{equation*}
QH_{\tau}(\mathcal{X}):=(H(\mathcal{X})\otimes_{\mathbb{Q}}\Lambda_{\mathcal{X}},\bullet_{\tau}).
\end{equation*}

\begin{defn}\label{Seidel element}
The Seidel element of $\mathcal{X}$ is the class 
\begin{equation}\label{seidelelement}
S(\hat{\tau}):=\sum\limits_{\alpha}\sum\limits_{\beta\in NE(E)^{sec}_{\mathbb{Z}}}\sum\limits_{l\geq0}\frac{1}{l!}\langle \textbf{1},\hat{\tau}_{tw},\ldots,\hat{\tau}_{tw},\imath_{*}\phi_{\alpha}\psi\rangle^{\mathcal{E}}_{0,l+2,\beta}\phi^{\alpha}e^{\langle \hat{\tau}_{0,2},\beta\rangle},
\end{equation}
in $QH_{\tau}(\mathcal{X})\otimes_{\Lambda_{\mathcal{X}}}\Lambda_{\mathcal{E}}$. 
Here $\imath:\mathcal{X}\rightarrow \mathcal{E}$ 
is the inclusion of a fiber, and
\begin{equation*}
\imath_{*}:H^{*}(\mathcal{IX};\mathbb{Q})\rightarrow H^{*+2}(\mathcal{IE};\mathbb{Q})
\end{equation*}
is the Gysin map. Moreover,
\begin{equation*}
e^{\langle \hat{\tau}_{0,2},\beta\rangle}= q^{\beta}=q_{0}^{\langle p_{0},\beta\rangle}\cdots q_{r}^{\langle p_{r},\beta \rangle},
\end{equation*} 
where 
\begin{equation*}
\hat{\tau}_{0,2}=\sum\limits_{a=0}^{r}p_{a}logq_{a}\in H^{2}(\mathcal{E})
\quad \text{and} \quad 
\hat{\tau}=\hat{\tau}_{0,2}+\hat{\tau}_{tw}\in H^{\leq 2}_{orb}(\mathcal{E}).
\end{equation*}
The Seidel element can be factorized as 
\begin{equation}
S(\hat{\tau})=q_{0}\tilde{S}(\hat{\tau}), \quad \text{with} \quad 
\tilde{S}(\hat{\tau})\in QH_{\tau}(\mathcal{X}).
\end{equation}
\end{defn}

\subsection{J-functions}

We will explain the relation between the Seidel element and 
the $J$-function of the associated bundle $\mathcal{E}$.

\begin{defn}

The $J$-function of $\mathcal{E}$ is the cohomology valued function 

\begin{equation}
J_{\mathcal{\mathcal{E}}}(\hat{\tau},z)=e^{\hat{\tau}_{0,2}/z}\left(1+\sum\limits_{\alpha}\sum\limits_{(\beta,l)\neq (0,0), \beta\in NE(E)_{\mathbb{Z}}}\frac{e^{\langle \hat{\tau}_{0,2},\beta\rangle}}{l!}\langle \textbf{1}, \hat{\tau}_{tw},\ldots,\hat{\tau}_{tw},\frac{\hat{\phi}_{\alpha}}{z-\psi}\rangle^{\mathcal{E}}_{0,l+2,\beta}\hat{\phi}^{\alpha}\right),
\end{equation}
where $\frac{\hat{\phi}_{\alpha}}{z-\psi}=\sum\limits_{n\geq 0}z^{-1-n}\hat{\phi}_{\alpha}\psi^{n}$.

\end{defn}

Note that when $n=0$, we will have 
\begin{enumerate}
\item $\sum\limits_{\alpha}\langle \textbf{1}, \hat{\tau}_{tw},\ldots,\hat{\tau}_{tw},\hat{\phi}_{\alpha}\rangle^{\mathcal{E}}_{0,l+2,\beta}\hat{\phi}^{\alpha}=0, \quad \text{for} \quad (l,\beta)\neq (1,0)$; 
\item $\sum\limits_{\alpha}\langle \textbf{1}, \hat{\tau}_{tw},\ldots,\hat{\tau}_{tw},\hat{\phi}_{\alpha}\rangle^{\mathcal{E}}_{0,l+2,\beta}\hat{\phi}^{\alpha}=\hat{\tau}_{tw}, \quad \text{for} \quad (l,\beta)=(1,0)$.
\end{enumerate}

The $J$-function can be expanded in terms of powers of $z^{-1}$ as follows:
 
\begin{equation}
J_{\mathcal{\mathcal{E}}}(\hat{\tau},z)=e^{\sum\limits_{a=0}^{r}p_{a}logq_{a}/z}\left(1+z^{-1}\hat{\tau}_{tw}+z^{-2}\sum\limits_{n=0}^{\infty}F_{n}(q_{1},\ldots,q_{r};\hat{\tau})q_{0}^{n}+O(z^{-3})\right),
\end{equation}
\noindent where 
\begin{equation}
F_{n}(q_{1},\ldots,q_{r};\hat{\tau})=\sum\limits_{\alpha=1}^{M}\sum\limits_{d\in NE(X)_{\mathbb{Z}}}\sum\limits_{l\geq0}\frac{1}{l!}\langle \textbf{1},\hat{\tau}_{tw},\ldots,\hat{\tau}_{tw},\hat{\phi}_{\alpha}\psi\rangle^{\mathcal{E}}_{0,l+2,d+n\sigma_{0}}q^{d}\hat{\phi}^{\alpha}
\end{equation}

\begin{prop}\label{J-function-Seidel}
The Seidel element corresponding to the $\mathbb{C}^{\times}$ action on $\mathcal{X}$  is given by 
\begin{equation}
S(\hat{\tau})=\imath^{*}\left(F_{1}(q_{1},\ldots,q_{r};\hat{\tau})q_{0}\right).
\end{equation}
\end{prop}
\begin{proof}
The proof in here is identical to the proof given in proposition 2.5 of \cite{GI} for smooth projective varieties:

Using the duality identity 
\begin{equation*}
\sum\limits_{\alpha=1}^{M}\hat{\phi}_{\alpha}\otimes \imath^{*}\hat{\phi}^{\alpha}=\sum\limits_{\alpha=1}^{N}\imath_{*}\phi_{\alpha}\otimes \phi^{\alpha},
\end{equation*}
we can see that
\begin{equation*}
\imath^{*}F_{1}(q_{1},\ldots,q_{r};\hat{\tau})=\sum\limits_{\alpha=1}^{N}\sum\limits_{d\in NE(X)_{\mathbb{Z}}}\sum\limits_{l\geq 0}\frac{1}{l!}\langle \textbf{1},\hat{\tau}_{tw},\ldots,\hat{\tau}_{tw},\imath_{*}\phi_{\alpha}\psi\rangle^{\mathcal{E}}_{0,l+2,d+\sigma_{0}}q^{d}\phi^{\alpha}.
\end{equation*}
Hence, the conclusion follows, i.e.
\begin{equation*}
S(\hat{\tau})=\imath^{*}(F_{1}(q_{1},\ldots,q_{r};\hat{\tau})q_{0}).
\end{equation*}
\end{proof}

\section{Seidel elements corresponding to toric divisors}

\subsection{A Review of Toric Deligne-Mumford stacks}

In this section, we will define toric Deligne-Mumford stacks 
following the construction of \cite{BCS} and \cite{Iritani}.

A toric Deligne-Mumford stack is defined by 
a stacky fan $\boldsymbol{\Sigma} =(\textbf{N},\Sigma,\beta)$, 
where $\textbf{N}$ is a finitely generated abelian group, 
$\Sigma \subset \textbf{N}_{\mathbb{Q}}=\textbf{N}\otimes_{\mathbb{Z}}\mathbb{Q}$ 
is a rational simplicial fan, and $\beta:\mathbb{Z}^{m}\rightarrow \textbf{N}$ 
is a homomorphism. We assume $\beta$ has finite cokernel and 
the rank of $\textbf{N}$ is $n$. 
The canonical map $\textbf{N} \rightarrow \textbf{N}_{\mathbb{Q}}$ 
generates the 1-skeleton of the fan $\Sigma$. 
Let $\bar{b_{i}}$ be the image of $b_{i}$ under this canonical map, where $b_{i}$ is the image under $\beta$ 
of the standard basis of 
$\mathbb{Z}^{m}$. Let $\mathbb{L} \subset \mathbb{Z}^{m}$ be 
the kernel of $\beta$. Then the fan sequence is the following exact sequence
\begin{equation}0
\longrightarrow \mathbb{L} \longrightarrow \mathbb{Z}^{m} \stackrel{\beta}{\longrightarrow} \textbf{N}.
\end{equation}
Let $\beta^{\vee}: (\mathbb{Z}^{*})^{m}\rightarrow \mathbb{L}^{\vee}$ 
be the Gale dual of $\beta$ in \cite{BCS}, 
where $\mathbb{L}^{\vee}:=H^{1}(Cone(\beta)^{*})$ 
is an extension of $\mathbb{L}^{*}=Hom(\mathbb{L},\mathbb{Z})$ 
by a torsion subgroup. The divisor sequence is the following exact sequence
\begin{equation}
0\longrightarrow\textbf{N}^{*}\stackrel{\beta^{*}}{\longrightarrow} (\mathbb{Z}^{*})^{m}\stackrel {\beta^{\vee}}{\longrightarrow}\mathbb{L}^{\vee}.
\end{equation}
By applying $Hom_{\mathbb{Z}}(-,\mathbb{C}^{\times})$ 
to the dual map $\beta^{\vee}$, 
we have a homomorphism 
\begin{equation*}
\alpha:G\rightarrow (\mathbb{C}^{\times})^{m}, \quad 
\text{where} \quad G:=Hom_{\mathbb{Z}}(\mathbb{L}^{\vee},\mathbb{C}^{\times}),
\end{equation*}
and we let $G$ act on $\mathbb{C}^{m}$ via this homomorphism.

The collection of anti-cones $\mathcal{A}$ is defined as follows: 
\[
\mathcal{A}:=\left\{I: \sum_{i\not\in I}\mathbb{R}_{\geq 0}\bar{b}_{i} \in \Sigma \right\}.
\]
 Let $\cU$ denote the open subset of $\mathbb{C}^{m}$ defined by $\mathcal{A}$:
\begin{equation*}
\mathcal{U}:=\mathbb{C}^{m}\setminus \cup_{I\not\in \mathcal{A}}\mathbb{C}^{I}, 
\end{equation*}
where
\[
\mathbb{C}^{I}=\left\{ (z_{1},\ldots,z_{m}):z_{i}=0
\text{ for } i \not\in I\right\}.
\]

\begin{defn}
Following \cite{Iritani}, the toric Deligne-Mumford stack $\mathcal{X}$ 
is defined as the quotient stack 
\[
\mathcal{X}:=[\mathcal{U}/G].
\]
\end{defn}

\begin{rmk}
The toric variety $X$ associated to the fan 
$\Sigma$ is the coarse moduli space of $\mathcal{X}$ \cite{BCS}.
\end{rmk}

\begin{defn} [\cite{Iritani}]
Given a stacky fan $\boldsymbol{\Sigma}=(\textbf{N},\Sigma,\beta)$, 
we define the set of box elements $Box(\boldsymbol{\Sigma})$ as follows
\[Box(\boldsymbol{\Sigma})=:\left\{ v\in \textbf{N}: \bar{v}=\sum\limits_{k\not\in I}c_{k}\bar{b}_{k}\text{ for some }0\leq c_{k}<1, I\in \mathcal{A}\right\}\]

\end{defn}

We assume that $\Sigma$ is complete, 
then the connected components of the inertia stack $\mathcal{IX}$ 
are indexed by the elements of 
$Box(\boldsymbol{\Sigma})$ (see \cite{BCS}). 
Moreover, given $v\in Box(\boldsymbol{\Sigma})$, the age of the corresponding connected component of $\mathcal{IX}$ is defined by  $age(v):=\sum\limits_{k\not\in I} c_{k}$. 

The Picard group $Pic(\mathcal{X})$ 
of $\mathcal{X}$
can be identified with the character group 
$Hom(G,\mathbb{C}^{\times})$. Hence
\begin{equation}
\mathbb{L}^{\vee}=Hom(G,\mathbb{C}^{\times})\cong Pic(\mathcal{X}) \cong H^{2}(\mathcal{X};\mathbb{Z}).
\end{equation}

We can also use the extended stacky fans introduced by Jiang \cite{Jiang} 
to define the toric Deligne-Mumford stacks. 
Given a stacky fan $\boldsymbol{\Sigma}=(\textbf{N},\Sigma,\beta)$ 
and a finite set 
\begin{equation*}
S=\{s_{1},\ldots,s_{l}\}\subset \textbf{N}_{\Sigma}:=\left\{c\in \textbf{N} : \bar{c} \in |\Sigma|\right\}. 
\end{equation*}
The $S$-extended stacky fan is given by 
$(\textbf{N},\Sigma, \beta^{S} )$, 
where $\beta^{S}: \mathbb{Z}^{m+l}\rightarrow \textbf{N}$ 
is defined by:
\begin{equation}
\beta^{S}(e_{i}) = \left\{
     \begin{array}{lr}
       b_{i} &  1\leq i \leq m;\\
       s_{i-m} &  m+1\leq i \leq m+l.
     \end{array}
   \right.
\end{equation}
Let $\mathbb{L}^{S}$ be the kernel of 
$\beta^{S}: \mathbb{Z}^{m+l}\rightarrow \textbf{N}$. 
Then we have the following $S$-extended fan sequence
\begin{equation}\label{S-ext-fan-seq}
0\longrightarrow \mathbb{L}^{S} \longrightarrow \mathbb{Z}^{m+l} \stackrel{\beta^{S}}{\longrightarrow} \textbf{N}.
\end{equation}
By the Gale duality, we have the $S$-extended divisor sequence
\begin{equation}
0\longrightarrow\textbf{N}^{*}\stackrel{\beta^{*}}{\longrightarrow} (\mathbb{Z}^{*})^{m+l}\stackrel {\beta^{S \vee}}{\longrightarrow}\mathbb{L}^{S \vee},
\end{equation}
where $\mathbb{L}^{S \vee}:=H^{1}(Cone(\beta^{S})^{*})$.
\begin{assumption}\label{S-ext-in-box}
In the rest of the paper, we will assume the set 
\begin{equation*}
\left\{v\in Box(\boldsymbol\Sigma); age(v)\leq 1\right\}\cup \left\{b_{1}\ldots,b_{m}\right\}
\end{equation*} 
generates $\textbf{N}$ over $\mathbb{Z}$. 
And we choose the set 
\begin{equation*}
S=\{s_{1},\ldots,s_{l}\}\subset Box(\boldsymbol\Sigma)
\end{equation*} 
such that the set $\{b_{1},\ldots,b_{m},s_{1},\ldots,s_{l}\}$ 
generates $\textbf{N}$ over $\mathbb{Z}$ 
and $age(s_{j})\leq 1$ for $1\leq j \leq l$. 
\end{assumption}
Let $D_{i}^{S}$ be the image of the standard basis of 
$(\mathbb{Z}^{*})^{m+l}$ under the map $\beta^{S\vee}$, 
then there is a canonical isomorphism
\begin{equation}\label{L^S-splitting}
\mathbb{L}^{S \vee}\otimes \mathbb{Q} \cong(\mathbb{L}^{\vee}\otimes\mathbb{Q})\bigoplus\limits_{i=m+1}^{m+l}\mathbb{Q}D_{i}^{S}, 
\end{equation}
which can be constructed as follows (\cite{Iritani}): 

Since $\Sigma$ is complete, for $m<j\leq m+l$, the box element $s_{j-m}$ is contained in some cone in $\Sigma$. 
Namely, 
\begin{equation*}
s_{j-m}=\sum_{i\not\in I^{S}_{j}}c_{ji}b_{i} \quad \text{in} 
\quad \textbf{N}\otimes \mathbb{Q}, \quad
c_{ji}\geq0, \quad \exists I^{S}_{j}\in \mathcal{A}^{S}, 
\end{equation*}
where $I^{S}_{j}$ is the "anticone" of the cone containing $s_{j-m}$.

By the $S$-extended fan sequence \ref{S-ext-fan-seq} tensored with $\mathbb{Q}$, 
we have the following short exact sequence
\begin{equation*}
0\longrightarrow \mathbb{L}^{S}\otimes\mathbb{Q}\longrightarrow \mathbb{Q}^{m+l}\stackrel{\beta^{S}}{\longrightarrow}\textbf{N}\otimes\mathbb{Q}\longrightarrow 0.
\end{equation*}
Hence, there exists a unique $D^{S\vee}_{j}\in\mathbb{L}^{S}\otimes\mathbb{Q}$ such that
\begin{equation}\label{D-dual}
\langle D^{S}_{i},D^{S\vee}_{j}\rangle = \left\{
     \begin{array}{lr}
       1 &   i=j;\\
       -c_{ji} &  i\not\in I^{S}_{j};\\
       0 &  i\in I^{S}_{j}\setminus \{j\}.
     \end{array}
   \right.
\end{equation}
These vectors $D^{S\vee}_{j}$ define a decomposition
\begin{equation*}
\mathbb{L}^{S\vee}\otimes\mathbb{Q}=\text{Ker}\left(\left(D^{S\vee}_{m+1},\ldots,D^{S\vee}_{m+l}\right):\mathbb{L}^{S\vee}\otimes\mathbb{Q}\rightarrow\mathbb{Q}^{l}\right)\oplus\bigoplus\limits_{j=m+1}^{m+l}\mathbb{Q}D^{S}_{j}.
\end{equation*}
We identify the first factor $\text{Ker}(D^{S\vee}_{m+1},\ldots,D^{S\vee}_{m+l})$
with $\mathbb{L}^{\vee}\otimes \mathbb{Q}$.
Via this decomposition, we can regard  $H^{2}(\mathcal{X},\mathbb{Q})\cong \mathbb{L}^{\vee}\otimes \mathbb{Q}$ as a subspace of $\mathbb{L}^{S\vee}\otimes\mathbb{Q}$.

Let $D_{i}$ be the image of $D_{i}^{S}$ in $\mathbb{L}^{\vee}\otimes \mathbb{Q}$ 
under this decomposition. Then 
\begin{equation*}
D_{i}=0, \quad \text{for}\quad 
m+1\leq i \leq m+l. 
\end{equation*}
Let $\mathcal{A}^{S}$ be the collection of $S$-extended anti-cones, i.e. 
\begin{equation*}
\mathcal{A}^{S}:=\left\{I^{S}: \sum_{i\not\in I^{S}}\mathbb{R}_{\geq 0}\overline{\beta^{S}(e_{i})}\in \Sigma\right\}. 
\end{equation*}
Note that 
\[
\{s_{1},\ldots,s_{l}\}\subset I^{S}, \quad \forall I^{S}\in \mathcal{A}^{S}.
\]

By applying $Hom_{\mathbb{Z}}(-,\mathbb{C}^{\times})$ to the $S$-extended dual map $\beta^{\vee}$, we have a homomorphism 
\begin{equation*}
\alpha^{S}:G^{S} \rightarrow (\mathbb{C}^{\times})^{m+l}, \quad \text{where}\quad G^{S}:=Hom_{\mathbb{Z}}(\mathbb{L}^{S \vee}, \mathbb{C}^{\times}).  
\end{equation*}
We define $\mathcal{U}$ to be the open subset of $\mathbb{C}^{m+l}$ defined by $\mathcal{A}^{S}$:
\begin{equation*}
\mathcal{U}^{S}:=\mathbb{C}^{m+l}\setminus \cup_{I^{S}\not\in \mathcal{A}^{S}}\mathbb{C}^{I^{S}}=\mathcal{U}\times (\mathbb{C}^{\times})^{l},
\end{equation*}
where 
\begin{equation*}
\mathbb{C}^{I^{S}}=\left\{(z_{1},\ldots, z_{m+l}):z_{i}=0\text{ for }i\not \in I^{S}\right\}. 
\end{equation*}
Let $G^{S}$ act on $\mathcal{U}^{S}$ via $\alpha^{S}$. 
Then we obtain the quotient stack $[\mathcal{U}^{S}/G^{S}]$. 
Jiang \cite{Jiang} showed that 
\begin{equation*}
[\mathcal{U}^{S}/G^{S}]\cong [\mathcal{U}/G]=\mathcal{X}.
\end{equation*}

\subsection{Mirror theorem for toric stacks}

In \cite{CCIT}, Coates-Corti-Iritani-Tseng defined the 
$S$-extended $I$-function of a smooth toric Deligne-Mumford stack 
$\mathcal{X}$ with projective coarse moduli space and 
proved that this $I$-function is a point of Givental's Lagrangian cone $\mathcal{L}$ 
for the Gromov-Witten theory of $\mathcal{X}$. 
In this paper, we will only need this theorem for the weak Fano case. In this case, the mirror theorem will take a particularly simple form which can be stated as an equality of $I$-function and $J$-function via a change of variables, called mirror map.

To state the mirror theorem for weak Fano toric Deligne-Mumford stack, we need the following definitions. 

We define the $S$-extended K\"ahler cone $C^{S}_{\mathcal{X}}$ as
\begin{equation*}
C^{S}_{\mathcal{X}}:=\cap_{I^{S}\in \mathcal{A}^{S}}\Sigma_{i\in I^{S}}\mathbb{R}_{>0}D_{i}^{S}
\end{equation*} 
and the K\"ahler cone $C_{\mathcal{X}}$ as
\begin{equation*}
C_{\mathcal{X}}:=\cap_{I \in \mathcal{A}}\Sigma_{i\in I}\mathbb{R}_{>0}D_{i}.
\end{equation*}
Let $p^{S}_{1},\ldots,p^{S}_{r+l}$ be an integral basis of $\mathbb{L}^{S\vee}$, 
where $r=m-n$, such that $p^{S}_{i}$ 
is in the closure $cl(C^{S}_{\mathcal{X}})$ of 
the $S$-extended K\"ahler cone $C^{S}_{\mathcal{X}}$ 
for all $1\leq i \leq r+l$ and $p^{S}_{r+1},\ldots, p^{S}_{r+l}$ 
are in $\sum\limits_{i=m+1}^{m+l}\mathbb{R}_{\geq 0}D^{S}_{i}$. 
We denote the image of $p^{S}_{i}$ in 
$\mathbb{L}^{\vee}\otimes \mathbb{R}$ by 
$p_{i}$, therefore $p_{1},\ldots,p_{r}$ 
are nef and $p_{r+1},\ldots,p_{r+l}$ are zero.
We define a matrix $(m_{ia})$ by 
\[
D^{S}_{i}=\sum\limits^{r+l}_{a=1}m_{ia}p^{S}_{a}, \quad m_{ia}\in \mathbb{Z}.
\]
Then the class $D_{i}$ of toric divisor is given by
\[
D_{i}=\sum\limits^{r}_{a=1}m_{ia}p_{a}.
\]

\begin{defn} [\cite{Iritani}, Section 3.1.4]

A toric Deligne-Mumford stack $\mathcal{X}$ 
is called weak Fano if the first Chern class $\rho$ satisfies
\begin{equation*}
\rho=c_{1}(T\mathcal{X})=\sum\limits_{i=1}^{m}D_{i} \in cl(C_{\mathcal{X}}),
\end{equation*} 
where $C_{\mathcal{X}}$ is the K\"ahler cone of $\mathcal{X}$.
\end{defn}

We will need a slightly stronger condition: 
\begin{equation*}
\rho^{S}:=D_{1}^{S}+\ldots+D_{m+l}^{S} \in cl(C^{S}_{\mathcal{X}}),
\end{equation*}
where $C_{\mathcal{X}}^{S}$ is the $S$-extended K\"ahler cone. By lemma 3.3 of \cite{Iritani}, we can see that $\rho^{S}\in cl(C^{S}_{\mathcal{X}})$ implies $\rho \in cl(C_{\mathcal{X}})$. Moreover, under assumption \ref{S-ext-in-box}, we will have 
\begin{equation*}
\rho^{S}\in cl(C^{S}_{\mathcal{X}})\quad \text{ if and only if} \quad \rho \in cl(C_{\mathcal{X}}).
\end{equation*}

For a real number $r$, let $\lceil r \rceil$, $\lfloor r \rfloor$ and $\{ r \}$ 
be the ceiling, floor and fractional part of $r$ respectively. 
\begin{defn}
We define two subsets $\mathbb{K}$ and 
$\mathbb{K}_{\text{eff}}$ of $\mathbb{L}^{S}\otimes \mathbb{Q}$ as follows:
\begin{equation*}\mathbb{K}:=\left\{d\in L^{S}\otimes \mathbb{Q}; \{ i \in \{ 1,\ldots,m+l\};\langle D^{S}_{i},d\rangle \in \mathbb{Z}\} \in \mathcal{A}^{S}\right\},
\end{equation*}
\begin{equation*}
\mathbb{K}_{\text{eff}}:=\left\{d\in L^{S}\otimes \mathbb{Q}; \{ i \in \{ 1,\ldots,m+l\};\langle D^{S}_{i},d\rangle \in \mathbb{Z}_{\geq 0}\} \in \mathcal{A}^{S}\right\}.
\end{equation*}
\end{defn}
\begin{rmk}
We will use $\mathbb{K}_{\mathcal{E}_{j}}$ and 
$\mathbb{K}_{\text{eff},\mathcal{E}_{j}}$ 
to denote the corresponding sets for 
the associated bundle $\mathcal{E}_{j}$, 
and use $\mathbb{K}_{\mathcal{X}}$ and 
$\mathbb{K}_{\text{eff},\mathcal{X}}$ 
to denote the corresponding sets for $\mathcal{X}$.
\end{rmk}
\begin{defn} [\cite{Iritani}, Section 3.1.3]
The reduction function $v$ is defined as follows:
\begin{align*}
v:\mathbb{K}& \longrightarrow Box(\boldsymbol{\Sigma})\\
d & \longmapsto \sum\limits_{i=1}^{m}\lceil \langle D^{S}_{i},d\rangle\rceil b_{i}+\sum\limits_{j=1}^{l}\lceil \langle D^{S}_{m+j},d\rangle\rceil s_{j} 
\end{align*}
\end{defn}

By the $S$-extended fan exact sequence, we have 
\begin{equation*}
\sum\limits_{i=1}^{m}\langle D^{S}_{i},d\rangle b_{i}+\sum\limits_{j=1}^{l}\langle D^{S}_{m+j},d\rangle s_{j}=0 \in \textbf{N}\otimes \mathbb{Q}.
\end{equation*}
Moreover, by the definition of $\mathbb{K}$, we have 
\[
\langle D_{m+j}^{S},d\rangle\in \mathbb{Z},\quad \text{for all}\quad d\in \mathbb{K}\quad \text{and} \quad 1\leq j \leq l.  
\]
Hence, 
\begin{equation*}
v(d)=\sum\limits_{i=1}^{m}\{-\langle D^{S}_{i},d\rangle \}b_{i}+\sum\limits_{j=1}^{l}\{-\langle D^{S}_{m+j},d\rangle\} s_{j}=\sum\limits_{i=1}^{m}\{-\langle D^{S}_{i},d\rangle \}b_{i}.
\end{equation*}

By abuse of notation, we use $D_{i}$ to denote the divisor 
$\{z_{i}=0\}\subset \mathcal{X}$ and 
the cohomology class in 
$H^{2}(\mathcal{X};\mathbb{Z})\cong \mathbb{L}^{\vee}$, 
for $1\leq i \leq m$. 

We consider the $\mathbb{C}^{\times}$-action 
fixing a toric divisor $D_{j}$, $1\leq j \leq m$, 
the action of $\mathbb{C}^{\times}$ on $\mathbb{C}^{m}$ is given by 
\[
(z_{1},\ldots,z_{m})\mapsto (z_{1},\ldots,t^{-1}z_{j},\ldots,z_{m}), \quad t\in \mathbb{C}^{\times}.
\]
We can extend this to the diagonal $\mathbb{C}^{\times}$-action 
on $\mathcal{U}\times (\mathbb{C}^{2}\setminus \{0\})$ by 
\[
(z_{1},\ldots,z_{m},u,v)\mapsto (z_{1},\ldots,t^{-1}z_{j},\ldots,z_{m},tu,tv), \quad t\in \mathbb{C}^{\times}.
\]
The associated bundle $\mathcal{E}_{j}$ 
of the $\mathbb{C}^{\times}$-action on $\mathcal{X}$ 
is given by 
\begin{equation*}
\mathcal{E}_{j}=\mathcal{U}\times(\mathbb{C}^{2}\setminus\{0\})/G\times\mathbb{C}^{\times}.
\end{equation*} 
We can also use the $S$-extended stacky fan of 
$\mathcal{X}$ to define $\mathcal{E}_{j}$: 
\begin{equation*}
\mathcal{E}_{j}=\mathcal{U^{S}}\times(\mathbb{C}^{2}\setminus\{0\})/G^{S}\times\mathbb{C}^{\times}.
\end{equation*}
Therefore $\mathcal{E}_{j}$ 
is also a toric Deligne-Mumford stack. 
We can identify $H^{2}(\mathcal{E}_{j};\mathbb{Z})$ 
with the lattice of the characters of $G\times \mathbb{C}^{\times}$: 
\begin{equation}\label{splitting-H^2}
H^{2}(\mathcal{E}_{j};\mathbb{Z})\cong \mathbb{L}^{\vee}\oplus\mathbb{Z}\cong H^{2}(\mathcal{X};\mathbb{Z})\oplus \mathbb{Z}.
\end{equation}
Moreover, we have the divisor sequence 
\begin{equation*}
0\rightarrow \textbf{N}^{*}\oplus \mathbb{Z}\rightarrow (\mathbb{Z}^{*})^{m+2}\rightarrow \mathbb{L}^{\vee}\oplus \mathbb{Z}. 
\end{equation*}
And the $S$-extended divisor sequence 
\begin{equation*}
0\rightarrow \textbf{N}^{*}\oplus \mathbb{Z}\rightarrow (\mathbb{Z}^{*})^{m+l+2}\rightarrow \mathbb{L}^{S \vee}\oplus \mathbb{Z}.
\end{equation*}
Let $\hat{D}^{S}_{i} $ be the image of the standard basis of $(\mathbb{Z}^{*})^{m+l+2}$ in $\mathbb{L}^{S\vee}\oplus \mathbb{Z}$. Then 
\begin{equation}
\hat{D}^{S}_{i}=(D^{S}_{i},0), \text{ for } i\neq j; \quad \hat{D}^{S}_{j}=(D^{S}_{j},-1);\quad \hat{D}_{m+l+1}^{S}=\hat{D}_{m+l+2}^{S}=(0,1).
\end{equation}
And, 
\begin{equation}
\hat{D}_{i}=(D_{i},0), \text{ for }i\neq j;\quad \hat{D}_{j}=(D_{j},-1); \quad \hat{D}_{m+1}=\hat{D}_{m+2}=(0,1).
\end{equation}
The fan $\Sigma_{j}$ of  $\mathcal{E}_{j}$ is a rational simplicial fan contained in $N_{\mathbb{Q}}\oplus \mathbb{Q}$. The 1-skeleton is given by 
\begin{equation}
\hat{b}_{i}=(b_{i},0),\text{ for }1\leq i \leq m;\quad \hat{b}_{m+1}=(0,1);\quad \hat{b}_{m+2}=(b_{j},-1).
\end{equation}
We set 
\begin{equation*}
p_{0}:=(0,1)=\hat{D}_{m+1}=\hat{D}_{m+2}\in H^{2}(\mathcal{E}_{j};\mathbb{Q}),
\end{equation*}
then a nef integral basis $\{p_{1},\ldots,p_{r}\}$ of $H^{2}(\mathcal{X};\mathbb{Q})$ 
can be lifted to a nef integral basis 
$\{p_{0},p_{1},\ldots,p_{r}\}$ of $H^{2}(\mathcal{E}_{j};\mathbb{Q})$, under the splitting (\ref{splitting-H^2}). 
Let $p_{1}^{S},\ldots,p_{r+l}^{S}$ 
be an integral basis of $\mathbb{L}^{S\vee}$, 
such that $p_{i}$ is the image of 
$p_{i}^{S}$ in $\mathbb{L}^{\vee}\otimes \mathbb{R}$. 
Let $p_{0}^{S}, p_{1}^{S},\ldots,p_{r+l}^{S}$ 
be an integral basis of $\mathbb{L}^{S\vee}\oplus \mathbb{Z}$ and 
$p_{0}$ is the image of 
\begin{equation*}
p_{0}^{S}=\hat{D}_{m+l+1}^{S}=\hat{D}_{m+l+2}^{S}
\end{equation*}
 in $(\mathbb{L}^{\vee}\oplus\mathbb{Z})\otimes \mathbb{R}$. 
Note that $p_{r+1},\ldots,p_{r+l}$ are zero. We have 
\[C_{\mathcal{E}_{j}}^{S}=C_{\mathcal{X}}^{S}+\mathbb{R}_{>0}p_{0}^{S},\quad \rho^{S}_{\mathcal{E}_{j}}=\rho^{S}_{\mathcal{X}}+p_{0}^{S}.\]
The following result is straightforward.
\begin{lemma}\label{weak-fano-lemma}
If $\rho^{S}_{\cX}\in cl(C^{S}_{\cX})$, then $\rho^{S}_{\cE_{j}}\in cl(C^{S}_{\cE_{j}})$, for $1\leq j \leq m$.
\end{lemma}
\begin{defn}
The $I$-function of $\cX$ is the $H^{*}_{orb}(\cX)$-valued function:
\begin{equation}
I_{\mathcal{X}}(y,z)=e^{\sum\limits_{i=1}^{r}p_{i}logy_{i}/z}
\sum\limits_{d\in \mathbb{K}_{\text{eff},\mathcal{X}}}
\prod\limits_{i=1}^{m+l}
\left(\frac{\prod_{k=\lceil \langle D^{S}_{i},d\rangle \rceil}^{\infty}
\left({D}_{i}+\left(\langle D^{S}_{i},d\rangle-k\right)z\right)}
{\prod_{k=0}^{\infty}\left({D}_{i}+\left(\langle D^{S}_{i},d\rangle-k\right)z\right)}\right)
y^{d}\textbf{1}_{v(d)},
\end{equation}
where $y^{d}=y_{1}^{\langle p_{1}^{S},d\rangle}\cdots y_{r+l}^{\langle p_{r+l}^{S},d\rangle}$. Similarly, The $I$-function of $\cE$ is the $H^{*}_{orb}(\cE)$-valued function:
\begin{equation}
 I_{\mathcal{E}_{j}}(y,z)=
e^{\sum\limits_{i=0}^{r}p_{i}logy_{i}/z}
\sum\limits_{\beta\in \mathbb{K}_{\text{eff},\mathcal{E}_{j}}}
\prod\limits_{i=1}^{m+l+2}
\left(\frac{\prod_{k=\lceil \langle \hat{D}^{S}_{i},\beta\rangle \rceil}^{\infty}
\left(\hat{D}_{i}+\left(\langle \hat{D}^{S}_{i},\beta\rangle-k\right)z\right)}
{\prod_{k=0}^{\infty}\left(\hat{D}_{i}+\left(\langle \hat{D}^{S}_{i},\beta\rangle-k\right)z\right)}
\right)y^{\beta}\textbf{1}_{v(\beta)},
\end{equation}
\noindent where  $y^{\beta}=y_{0}^{\langle p_{0}^{S},\beta,\rangle}y_{1}^{\langle p_{1}^{S},\beta\rangle}\cdots y_{r+l}^{\langle p_{r+l}^{S},\beta\rangle}$.

\end{defn}

Following section 4.1 of \cite{Iritani}, The $I$-functions of $\mathcal{X}$ and $\mathcal{E}_{j}$ can be rewritten in the form:

\begin{equation}
I_{\mathcal{X}}(y,z)=e^{\sum\limits_{i=1}^{r}p_{i}logy_{i}/z}
\sum\limits_{d\in \mathbb{K}_{\mathcal{X}}}
\prod\limits_{i=1}^{m+l}
\left(\frac{\prod_{k=\lceil \langle D^{S}_{i},d\rangle \rceil}^{\infty}
\left({D}_{i}+\left(\langle D^{S}_{i},d\rangle-k\right)z\right)}
{\prod_{k=0}^{\infty}\left({D}_{i}+\left(\langle D^{S}_{i},d\rangle-k\right)z\right)}\right)
y^{d}\textbf{1}_{v(d)},
\end{equation}
and
\begin{equation}\label{I-function-seidel-space-with-K}
 I_{\mathcal{E}_{j}}(y,z)=
e^{\sum\limits_{i=0}^{r}p_{i}logy_{i}/z}
\sum\limits_{\beta\in \mathbb{K}_{\mathcal{E}_{j}}}
\prod\limits_{i=1}^{m+l+2}
\left(\frac{\prod_{k=\lceil \langle \hat{D}^{S}_{i},\beta\rangle \rceil}^{\infty}
\left(\hat{D}_{i}+\left(\langle \hat{D}^{S}_{i},\beta\rangle-k\right)z\right)}
{\prod_{k=0}^{\infty}\left(\hat{D}_{i}+\left(\langle \hat{D}^{S}_{i},\beta\rangle-k\right)z\right)}
\right)y^{\beta}\textbf{1}_{v(\beta)},
\end{equation}
respectively, because the summand with $d\in \mathbb{K}\setminus \mathbb{K}_{\text{eff}}$ vanishes. We refer to \cite{Iritani} for more details.

\begin{thm}[\cite{Iritani}, Conjecture 4.3]\label{mirror-theorem}
Assume that $\rho^{S} \in cl(C^{S}_{\mathcal{X}})$. 
Then the $I$-function and the $J$-function satisfy the following relation:
\begin{equation}
I_{\mathcal{X}}(y,z)=J_{\mathcal{X}}(\tau(y),z)
\end{equation}
where 
\begin{equation}
\tau(y)=\tau_{0,2}(y)+\tau_{tw}(y)=\sum\limits_{i=1}^{r}(logy_{i})p_{i}+\sum\limits_{j=m+1}^{m+l}y^{D^{S\vee}_{j}}\mathfrak{D}_{j}+h.o.t.\in H^{\leq 2}_{orb}(\mathcal{X}),
\end{equation}
with 
\[\tau_{0,2}(y)\in  H^{2}(\mathcal{X}), \quad \tau_{tw}(y)\in H^{\leq 2}_{orb}(\mathcal{X})\setminus H^{2}(\mathcal{X}),
\]
\[
\mathfrak{D}_{j}=\prod\limits_{i\not\in I_{j}}D_{i}^{\lfloor c_{ji} \rfloor}\boldsymbol{1}_{v(D^{S\vee}_{j})}\in H^{*}_{orb}(\mathcal{X}).
\]
and $h.o.t.$ stands for higher order terms in $z^{-1}$.
Furthermore, $\tau(y)$ is called the mirror map and takes values in $H^{\leq 2}_{orb}(\mathcal{X})$.
\end{thm}

For $\tau_{0,2}(y)=\sum\limits_{a=1}^{r}p_{a}logq_{a}\in H^{2}(\mathcal{X})$, we have 
\begin{equation*}
logq_{i}=logy_{i}+g_{i}(y_{1},\ldots,y_{r+l}),\text{ for }i=1,\ldots,r,
\end{equation*} 
where $g_{i}$ is a (fractional) power series in $y_{1},\ldots,y_{r+l}$ 
which is homogeneous of degree zero with respect to 
the degree $degy^{d}=2\langle\rho^{S}_{\mathcal{X}},d\rangle$.

By lemma \ref{weak-fano-lemma}, under the assumption of theorem \ref{mirror-theorem}, we can also apply the mirror theorem to the associated bundle $\cE_{j}$, hence we have
\begin{equation*}
I_{\mathcal{E}_{j}}(y,z)=J_{\mathcal{E}_{j}}(\tau^{(j)}(y),z),
\end{equation*}
where
\[
\tau^{(j)}(y)=\tau^{(j)}_{0,2}+\tau^{(j)}_{tw}(y)\in H^{2}(\cE_{j})\oplus\left(H^{\leq 2}_{orb}(\cE_{j})\setminus H^{2}(\cE_{j})\right)
\]
Since $\tau^{(j)}_{0,2}(y)=\sum\limits_{a=0}^{r}p_{a}logq_{a}\in H^{2}(\mathcal{E}_{j})$, therefore 
\begin{equation*}
logq_{i}=logy_{i}+g^{(j)}_{i}(y_{0},\ldots,y_{r+l}),\text{ for }i=0,\ldots,r,
\end{equation*} 
where $g^{(j)}_{i}$ is a (fractional) power series in $y_{0},y_{1},\ldots,y_{r+l}$ 
which is homogeneous of degree zero with respect to 
the degree $degy^{\beta}=2\langle\rho^{S}_{\mathcal{E}_{j}},\beta\rangle$.

\subsection{Seidel elements and mirror maps}
\begin{proposition}\label{g^(j)-indep-y_0}
The function $g_{i}^{(j)}$ does not depend on $y_{0}$ and we have 
\begin{equation*}
g_{i}^{(j)}(y_{0},\ldots,y_{r+l})=g_{i}(y_{1},\ldots,y_{r+l}),\quad \text{for}\quad i=1,\ldots,r.
\end{equation*}
\end{proposition}
\begin{proof}
The functions $g_{i}$ is the coefficients of $z^{-1}p_{i}$ in the expansion of $I_{\mathcal{X}}$: 
\begin{equation*}
I_{\mathcal{X}}(y,z)=e^{\sum\limits_{i=1}^{r}p_{i}logy_{i}/z}\left(1+z^{-1}\left(\sum\limits_{i=1}^{r}g_{i}(y)p_{i}+\tau_{tw}\right)+O(z^{-2})\right).
\end{equation*}
The functions $g_{i}^{(j)}$ is the coefficients of $z^{-1}p_{i}$ in the expansion of $I_{\mathcal{E}_{j}}$: 
\begin{equation*}
I_{\mathcal{E}_{j}}(y,z)=e^{\sum\limits_{i=0}^{r}p_{i}logy_{i}/z}\left(1+z^{-1}\left(\sum\limits_{i=0}^{r}g_{i}^{(j)}(y)p_{i}+\tau^{(j)}_{tw}\right)+O(z^{-2})\right).
\end{equation*}
Following the proof of lemma 3.5 of \cite{GI}, we obtain the conclusion of this proposition. 
\end{proof}

We will prove $\tau^{(j)}_{tw}$ 
is also independent from $y_{0}$. 
To begin with, the following lemma implies that 
$\tau_{tw}^{(j)}(y)$ is an (integer) power series in $y_{0}$.

\begin{lemma}
For any $\beta\in \mathbb{K}_{\mathcal{E}_{j}}$, 
we have $\langle p^{S}_{0},\beta\rangle\in \mathbb{Z}$. 
Furthermore, for any  $\beta\in \mathbb{K}_{\text{eff}, \mathcal{E}_{j}}$, 
we have $\langle p^{S}_{0},\beta\rangle\in \mathbb{Z}_{\geq0}$.
\end{lemma}
\begin{proof}
Any cone $\sigma\in \Sigma_{j}$ 
containing both $\hat{b}_{m+1}$ and 
$\hat{b}_{m+2}$ should also contain $\hat{b}_{j}$, 
this is impossible since the fan $\Sigma_{j}$ is simplicial and 
$\hat{b}_{m+1}$, $\hat{b}_{m+2}$ and 
$\hat{b}_{j}$ lie in the same plane. 
Hence, by the definition of $\mathbb{K}_{\mathcal{E}_{j}}$ 
(resp. $\mathbb{K}_{\text{eff}, \mathcal{E}_{j}}$), 
at least one of $\langle \hat{D}^{S}_{m+1},\beta\rangle$ and 
$\langle \hat{D}^{S}_{m+2},\beta\rangle$ has to be integer (resp. non-negative integer), 
 for any $\beta\in\mathbb{K}_{\mathcal{E}_{j}}$ (resp. $\beta\in \mathbb{K}_{\text{eff}, \mathcal{E}_{j}}$). 
On the other hand, we have, 
\[
\langle p^{S}_{0},\beta\rangle=\langle \hat{D}^{S}_{m+1},\beta\rangle=\langle \hat{D}^{S}_{m+2},\beta\rangle.
\]  
Therefore, we must have $\langle p^{S}_{0},\beta\rangle\in \mathbb{Z}$ (resp. $\langle p^{S}_{0},\beta\rangle\in \mathbb{Z}_{\geq0}$).
\end{proof}
As a direct consequence of the above lemma, $\tau_{tw}^{(j)}(y)$ can only contain non-negative integer power of $y_{0}$.

\begin{proposition}\label{tau-indep-y_0}
Let $\tau^{(j)}_{tw}(y)=\sum\limits_{n=0}^{\infty}H_{n}^{(j)}(y)y_{0}^{n}$, 
where $H_{n}^{(j)}(y)$ is a (fractional) power series in $y_{1},\ldots,y_{n}$.
Then 
\[
H_{n}^{(j)}(y)=0\quad \text{for}\quad n\geq 1,
\] 
i.e. $\tau^{(j)}_{tw}(y)$ is independent from $y_{0}$. Moreover, we have 
\[
\tau^{(j)}_{tw}(y)=\tau_{tw}(y).
\]
\end{proposition}

\begin{proof}

Recall $\tau^{(j)}_{tw}(y)$ is the coefficient of $z^{-1}$ in 
\begin{equation}
e^{-\sum\limits_{i=0}^{r}p_{i}logy_{i}/z}I_{\mathcal{E}_{j}}(y,z)=
\sum\limits_{\beta\in \mathbb{K}_{\text{eff},\mathcal{E}_{j}}}
\prod\limits_{i=1}^{m+l+2}
\left(\frac{\prod_{k=\lceil \langle \hat{D}^{S}_{i},\beta\rangle \rceil}^{\infty}
\left(\hat{D}_{i}+\left(\langle \hat{D}^{S}_{i},\beta\rangle-k\right)z\right)}
{\prod_{k=0}^{\infty}\left(\hat{D}_{i}+\left(\langle \hat{D}^{S}_{i},\beta\rangle-k\right)z\right)}
\right)y^{\beta}\textbf{1}_{v(\beta)},
\end{equation}
valued in $H^{\leq 2}_{orb}(\mathcal{E}_{j})\setminus H^{2}(\mathcal{E}_{j})$. 
Hence, we only need to consider terms with $v(\beta)\neq 0$, or, equivalently, $v(d)\neq 0$, where $d$ is the natural projection of $\beta$ on to $\mathbb{K}_{\text{eff},\mathcal{X}}$.

Therefore, it remains to examine the product factor:
\begin{align}\label{exp-of-prod-factor}
\notag & 
\prod\limits_{i=1}^{m+l+2}
\left(\frac{\prod_{k=\lceil \langle \hat{D}^{S}_{i},\beta\rangle \rceil}^{\infty}
\left(\hat{D}_{i}+\left(\langle \hat{D}^{S}_{i},\beta\rangle-k\right)z\right)}
{\prod_{k=0}^{\infty}\left(\hat{D}_{i}+\left(\langle \hat{D}^{S}_{i},\beta\rangle-k\right)z\right)}
\right)\\
\notag =& 
\frac{\prod_{i:\langle \hat{D}_{i}^{S},\beta\rangle<0}
\prod_{\langle \hat{D}_{i}^{S},\beta\rangle\leq k <0}
\left(\hat{D}_{i}+\left(\langle \hat{D}_{i}^{S},\beta\rangle-k\right)z\right)}
{\prod_{i:\langle \hat{D}_{i}^{S},\beta\rangle>0}
\prod_{0\leq k<\langle \hat{D}_{i}^{S},\beta\rangle}
\left(\hat{D}_{i}+\left(\langle \hat{D}_{i}^{S},\beta\rangle-k\right)z\right)}
\notag\\
=&
C_{\beta}z^{-\left(\sum_{i=1}^{m+l+2}\lceil \langle \hat{D}^{S}_{i},\beta\rangle \rceil +\# \{i: \langle \hat{D}_{i}^{S},\beta\rangle \in \mathbb{Z}_{<0}\}\right)}\prod_{i:\langle \hat{D}^{S}_{i},\beta\rangle \in \mathbb{Z}_{<0}}\hat{D}_{i}+h.o.t.,
\end{align}
where 
\begin{equation}
C_{\beta}=\prod\limits_{i:\langle \hat{D}^{S}_{i},\beta\rangle <0}\prod\limits_{\langle \hat{D}^{S}_{i},\beta\rangle <k<0}\left(\langle \hat{D}^{S}_{i},\beta\rangle -k\right)\prod\limits_{i:\langle \hat{D}^{S}_{i},\beta\rangle >0}\prod\limits_{0\leq k<\langle \hat{D}^{S}_{i},\beta\rangle}\left(\langle \hat{D}^{S}_{i},\beta\rangle -k\right)^{-1}.
\end{equation}

By assumption, we need to have 
\[
\sum_{i=1}^{m+l+2}\lceil \langle \hat{D}^{S}_{i},\beta\rangle \rceil \geq \sum_{i=1}^{m+l+2}\langle \hat{D}^{S}_{i},\beta\rangle\geq 0.
\] 
The equality holds if and only if 
\[
\langle \hat{D}^{S}_{i},\beta\rangle\in\mathbb{Z}, \quad \text{for all} \quad 1\leq i \leq m+l+2; \quad \text{and} \quad \sum_{i=1}^{m+l+2}\langle \hat{D}^{S}_{i},\beta\rangle=0.
\]
However, this would imply $v(\beta)=0$, hence we cannot have $ \sum_{i=1}^{m+l+2}\lceil \langle \hat{D}^{S}_{i},\beta\rangle \rceil=0$. Therefore, the expansion (\ref{exp-of-prod-factor}) would contribute to $H_{n}^{(j)}$ only when 
\begin{equation*}
\sum_{i=1}^{m+l+2}\lceil \langle \hat{D}^{S}_{i},\beta\rangle \rceil=1\quad \text{and}\quad\# \{i: \langle \hat{D}_{i}^{S},\beta\rangle \in \mathbb{Z}_{<0}\}=0.
\end{equation*} 
In this case,  if $\langle p^{S}_{0},\beta\rangle\geq 1$, then 
\[
\sum_{i=1}^{m+l+2}\lceil \langle \hat{D}^{S}_{i},\beta\rangle \rceil\geq\sum_{i=1}^{m+l}\lceil\langle D^{S}_{i},d\rangle\rceil+1,
\] 
therefore, we have 
\[
0\geq\sum_{i=1}^{m+l}\lceil\langle D^{S}_{i},d\rangle\rceil\geq\sum_{i=1}^{m+l}\langle D^{S}_{i},d\rangle=0.
\] 
This implies, when $\langle p^{S}_{0},\beta\rangle\geq1$, we must have
\[
\langle D^{S}_{i},d\rangle \in \mathbb{Z}, \text{ for }1\leq i \leq m+l.
\] 
It is a contradiction, since $\hat{\tau}_{tw}\in  H^{\leq 2}_{orb}(\mathcal{E}_{j})\setminus H^{2}(\mathcal{E}_{j})$ 
implies $v(d)\neq 0$. 
Hence 
\[
H_{n}^{(j)}=0\text{ for all }n>0
\] 
and $\tau^{(j)}_{tw}(y)$ is independent from $y_{0}$. 
Moreover, by the expression of $I$-functions and the identity
\[
\imath^{*}I_{\cE_{j}}\big |_{y_{0}=0}=I_{\cX},
\]
 we have $\tau^{(j)}_{tw}(y)=\tau_{tw}(y)$.

\end{proof}

As a direct consequence of the above lemma, we can use the following notation for the Seidel element 
\begin{equation}
\tilde{S}_{j}(\tau(y)):=\tilde{S}_{j}(\tau^{(j)}(y)), 
\end{equation}
since $\tilde{S}_{j}(\tau^{(j)}(y))$ does not depend on $y_{0}$ or $q_{0}$.

\subsection{Seidel Elements in terms of $I$-functions}

We can rewrite the $I$-function of the associated bundle $\mathcal{E}_{j}$ as follows:
\begin{equation}
 e^{\sum\limits_{i=0}^{r}p_{i}logy_{i}/z}\left(1+z^{-1}\left(\sum\limits_{i=0}^{r}g_{i}^{(j)}(y)p_{i}+\tau^{(j)}_{tw}(y)\right)+z^{-2}\left(\sum\limits_{n=0}^{2}G_{n}^{(j)}(y)y_{0}^{n}\right)+O(z^{-3})\right).
\end{equation}
Then, $logq_{i}=logy_{i}+g_{i}^{(j)}(y)$ implies 
\begin{equation}
I_{\mathcal{E}_{j}}(y,z)=e^{\sum\limits_{i=0}^{r}p_{i}logq_{i}/z}\left(1+z^{-1}\tau^{(j)}_{tw}(y)+z^{-2}\left(\sum\limits_{n=0}^{2}G_{n}^{(j)}(y)y_{0}^{n}\right)+O(z^{-3})\right),
\end{equation}
where $G_{n}^{(j)}(y)$ is a (fractional) power series in $y_{1},\ldots,y_{r+l}$ taking values in $H^{*}_{orb}(\mathcal{E}_{j})$.

By proposition (\ref{J-function-Seidel}), the Seidel element $\tilde{S}_{j}(\tau^{(j)}(y))$ 
is the coefficient of $q_{0}/z^{2}$ in 
\[
exp\left(-\sum\limits_{i=0}^{r}p_{i}logq_{i}/z\right)J_{\mathcal{E}_{j}}(\tau^{(j)}(y),z),
\] 
hence $J_{\mathcal{E}_{j}}(\tau^{(j)}(y),z)=I_{\mathcal{E}_{j}}(y,z)$ and $logq_{0}=logy_{0}+g_{0}^{(j)}(y)$ imply the following result:

\begin{thm}\label{Seidel-I-function}
The Seidel element $S_{j}$ associated to the toric divisor $D_{j}$ is given by 
\begin{equation}
S_{j}(\tau^{(j)}(y))=\imath^{*}(G_{1}^{(j)}(y)y_{0}).
\end{equation}
Furthermore, we have
\begin{equation}
\tilde{S}_{j}(\tau(y))=\tilde{S}_{j}(\tau^{(j)}(y))=exp(-g_{0}^{j}(y))\imath^{*}(G_{1}^{(j)}(y)).
\end{equation}
\end{thm}

\subsection{Computation of $g_{0}^{(j)}$}\label{computation-of-g}

The computation is essentially the same as the proof of lemma 3.16 of \cite{GI}. Consider the product factors in $I_{\mathcal{E}_{j}}$:
\begin{equation*}
\prod\limits_{i=1}^{m+l+2}\left(\frac{\prod_{k=\lceil \langle \hat{D}^{S}_{i},\beta\rangle \rceil}^{\infty}\left(\hat{D}_{i}+\left(\langle \hat{D}^{S}_{i},\beta\rangle-k\right)z\right)}{\prod_{k=0}^{\infty}\left(\hat{D}_{i}+\left(\langle \hat{D}^{S}_{i},\beta\rangle-k\right)z\right)}\right)y^{\beta}\textbf{1}_{v(\beta)},
\end{equation*}
these factors contribute to $g_{i}^{(j)}$ if 
\[
v(\beta)=\sum\limits_{i=1}^{m+l+2}\{-\langle \hat{D}_{i}^{S},\beta\rangle\}\hat{b}_{i}=0,
\]
 then, by the definition of $\mathbb{K}_{\text{eff}}$, we must have 
\[
\langle \hat{D}_{i}^{S},\beta\rangle\in \mathbb{Z},\text{ for all }1\leq i \leq m+l+2.
\]
 In this case, the product factors can be rewritten as
\begin{align}\label{expansion-C^beta}
\notag & \prod\limits_{i=1}^{m+l+2}\left(\frac{\prod_{k=\lceil \langle \hat{D}^{S}_{i},\beta\rangle \rceil}^{\infty}\left(\hat{D}_{i}+\left(\langle \hat{D}^{S}_{i},\beta\rangle-k\right)z\right)}{\prod_{k=0}^{\infty}\left(\hat{D}_{i}+\left(\langle \hat{D}^{S}_{i},\beta\rangle-k\right)z\right)}\right)y^{\beta}\textbf{1}_{v(\beta)}\\
\notag & =\prod\limits_{i=1}^{m+l+2}\frac{\prod_{k=-\infty}^{0}\left(\hat{D}_{i}+kz\right)}{\prod_{k=-\infty}^{\langle \hat{D}_{i}^{S},\beta\rangle}\left(\hat{D}_{i}+kz\right)}y^{\beta} \\
& = \left(C_{\beta}z^{-\sum_{i=1}^{m+l+2}\langle \hat{D}^{S}_{i},\beta\rangle -\#\{i:\langle \hat{D}_{i}^{S},\beta\rangle<0\}}\prod\limits_{i:\langle \hat{D}_{i}^{S},\beta\rangle <0}\hat{D}_{i}+h.o.t.\right)y^{\beta}, 
\end{align}
where $h.o.t.$ stands for higher order terms in $z^{-1}$ and 
\begin{equation}
C_{\beta}=\prod\limits_{i:\langle \hat{D}_{i}^{S},\beta\rangle<0}(-1)^{-\langle \hat{D}_{i}^{S},\beta\rangle-1}\left(-\langle \hat{D}_{i}^{S},\beta\rangle -1\right)!\prod\limits_{i:\langle \hat{D}_{i}^{S},\beta\rangle \geq 0}\left(\langle \hat{D}_{i}^{S},\beta\rangle !\right)^{-1}.
\end{equation}
They contribute to the $z^{-1}$ term if 
\[
\sum_{i=1}^{m+l+2}\langle \hat{D}^{S}_{i},\beta\rangle +\#\{i:\langle \hat{D}_{i}^{S},\beta\rangle<0\}\leq1.
\]
Since we assume $\rho^{S}_{\mathcal{X}}\in cl(C^{S}_{\mathcal{X}})$, 
hence $\rho^{S}_{\mathcal{E}_{j}}\in cl(C^{S}_{\mathcal{E}_{j}})$.
So it has to be the following three cases:
\begin{itemize}
\item $\left\{
        \begin{array}{lr}
         \sum_{i=1}^{m+l+2} \langle \hat{D}^{S}_{i},\beta\rangle=0\\
          \# \{i: \langle \hat{D}_{i}^{S},\beta\rangle \in \mathbb{Z}_{<0}\}=0
        \end{array}
       \right.$

\item 
 $\left\{
     \begin{array}{lr}
     \sum_{i=1}^{m+l+2}\langle \hat{D}^{S}_{i},\beta\rangle = 1 \\
      \# \{i: \langle \hat{D}_{i}^{S},\beta\rangle \in \mathbb{Z}_{<0}\}=0
     \end{array}
   \right. $
\item
$\left\{
     \begin{array}{lr}
     \sum_{i=1}^{m+l+2}\langle \hat{D}^{S}_{i},\beta\rangle \rceil= 0 \\
      \# \{i: \langle \hat{D}_{i}^{S},\beta\rangle \in \mathbb{Z}_{<0}\}=1
     \end{array}
   \right. $.
\end{itemize}
In the first case, we have $\langle \hat{D}^{S}_{i},\beta\rangle=0$ for all $i$, hence $\beta=0$; the second case can not happen, since $\beta$ has to satisfy $\langle \hat{D}^{S}_{i},\beta\rangle=0$ except for one $i$ and this implies $\beta=0$. Therefore, the coefficient of $z^{-1}$ is from the third case, where 
\begin{equation}\label{third-case}
\sum_{i=1}^{m+l+2}\langle \hat{D}^{S}_{i},\beta\rangle=0\quad \text{and}\quad \#\{i:\langle \hat{D}_{i}^{S},\beta\rangle<0\}=1.
\end{equation}
By the assumption $\rho^{S}_{\mathcal{X}}\in cl(C^{S}_{\mathcal{X}})$, 
we must have $\sum\limits_{i=1}^{m+l}\langle D^{S}_{i},d\rangle= 0$ and $\langle p^{S}_{0},\beta\rangle=0$. 
Moreover, $\langle D^{S}_{i},d\rangle<0$ for exactly one $i$ in $\{1,\ldots,m\}$. 
(Note that $\langle D^{S}_{i},d\rangle\geq 0$ for $i\in \{m+1,\ldots, m+l\}$.)

Now $g_{0}^{(j)}$ is the coefficient corresponding to $p_{0}$ and $\hat{D}_{j}=\langle D_{j},-1\rangle=D_{j}-p_{0}$ is the only one, among $\hat{D}_{1},\ldots,\hat{D}_{m}$, which contains $p_{0}$. By expression (\ref{expansion-C^beta}), we must have $\langle D^{S}_{j},d\rangle<0$ and $\langle D^{S}_{i},d\rangle\geq 0$ for $i\neq j$. Hence we have 

\begin{lemma}\label{expression-of-g}
The coefficient $g_{0}^{(j)}$ is given by
\begin{equation}\label{correction coefficient 1}
g_{0}^{j}(y_{1},\ldots,y_{r+l})=\sum\limits_{\substack{\langle D^{S}_{i},d\rangle\in \mathbb{Z}, 1 \leq i\leq m+l\\ \langle \rho^{S}_{\mathcal{X}},d\rangle=0\\ \langle D_{j}^{S},d\rangle <0\\ \langle D^{S}_{i},d\rangle \geq 0, \forall i \neq j}}\frac{(-1)^{-\langle D_{j}^{S},d\rangle}\left(-\langle D_{j}^{S},d\rangle-1\right)!}{\prod_{i\neq j}\langle D_{i}^{S},d\rangle !}y^{d}.
\end{equation}
\end{lemma}
\section{Batyrev Elements}

In this section, we will extend the definition of the Batyrev elements in \cite{GI} to toric Deligne-Mumford stacks and explore their relationships with the Seidel elements. 

\subsection{Batyrev Elements}

Following \cite{Iritani}, consider the mirror coordinates 
$y_{1},\ldots,y_{r+l}$ of the toric Deligne-Mumford stacks 
$\mathcal{X}$ with 
$\rho^{S}_{\mathcal{X}}\in cl(C^{S}_{\mathcal{X}})$. 
Set $\mathbb{C}[y^{\pm}]=\mathbb{C}[y_{1}^{\pm},\ldots,y_{r+l}^{\pm}]$. 

\begin{defn}
The Batyrev ring $B(\mathcal{X})$ of $\mathcal{X}$ is a 
$\mathbb{C}[y^{\pm}]$-algebra generated by 
the variables $\lambda_{1},\ldots,\lambda_{r+l}$ 
with the following two relations: 
\begin{align}
\label{eq:bat-rel}
\begin{split}
\text{(multiplicative):}& \qquad
  y^{d}\prod_{i:\langle D^{S}_{i},d\rangle<0}\omega_{i}^{-\langle D^{S}_{i},d\rangle}=\prod_{i:\langle D^{S}_{i},d\rangle >0}\omega_{i}^{\langle D^{S}_{i},d\rangle}, \quad d\in \mathbb{L^{S}};\\
\text{(linear):}& \qquad \omega_{i}=\sum\limits_{a=1}^{r+l}m_{ai}\lambda_{a},
\end{split}
\end{align} 
where $\omega_{i}$ is invertible in $B(\mathcal{X})$.
\end{defn}

\begin{defn}
We define the element $\tilde{p}_{i}^{S}\in H^{\leq 2}_{orb}(\cX)\otimes \mathbb{Q}[[y_{1},\ldots,y_{r+l}]]$
as 
\begin{equation*}
\tilde{p}_{i}^{S}=\frac{\partial \tau(y)}{\partial log y_{i}}, \quad i=1,\ldots,r+l.
\end{equation*}
Recall that 
\[
D_{j}^{S}=\sum\limits_{i=1}^{r+l}m_{ij}p^{S}_{i},\text{ for }1\leq j \leq m+l,
\] 
Then, the Batyrev element associated to $D_{j}^{S}$ is defined by 
\begin{equation*}
\tilde{D}^{S}_{j}=\sum\limits_{i=1}^{r+l}m_{ij}\tilde{p}^{S}_{i},\text{ for }1\leq j \leq m+l.
\end{equation*}
\end{defn}

\begin{proposition}
The Batyrev elements $\tilde{D}^{S}_{1},\ldots,\tilde{D}^{S}_{m+l}$ satisfy the multiplicative and linear Batyrev relations for $\omega_{j}=\tilde{D}^{S}_{j}$.
\end{proposition}
\begin{proof}
We consider the differential operator $\mathcal{P}_{d}\in \mathbb{C}[z,y^{\pm},zy(\partial/\partial y)]$ for $d\in \mathbb{L}^{S}$,
 introduced by Iritani in \cite{Iritani}, section 4.2:
\begin{equation}
\mathcal{P}_{d}:=y^{d}\prod\limits_{i:\langle D^{S}_{i},d\rangle <0}\prod\limits_{k=0}^{-\langle D^{S}_{i},d\rangle -1}(\mathcal{D}_{i}-kz)-\prod\limits_{i:\langle D^{S}_{i},d\rangle >0}\prod\limits^{\langle D^{S}_{i},d\rangle -1}_{k=0}(\mathcal{D}_{i}-kz),
\end{equation}
where $\mathcal{D}_{i}:=\sum\limits^{r+l}_{j=1}m_{ij}zy_{j}\partial/\partial y_{j}$.

By \cite{Iritani} lemma 4.6, we have
\[
\mathcal{P}_{d}I(y,z)=0,\quad d\in\mathbb{L}^{S}.
\] 
Hence 
\[
0=\mathcal{P}_{d}(z,y,zy\partial/\partial y)I(y,z)=\mathcal{P}_{d}(z,y,zy\partial/\partial y)J(\tau(y),z).
\]
This implies that 
\[
\mathcal{P}_{d}(z,y,z\tau^{*}\nabla)\boldsymbol{1}=0,
\] 
where $\tau^{*}\nabla_{i}:=\nabla_{\tau_{*}(y_{i}(\partial/\partial y_{i}))}$. 
Since 
\[
\tau(y)=\sum\limits_{i=1}^{r}p_{i}logy_{i}+\tau_{tw}(y)\quad \text{and} \quad \nabla_{\tau_{*}(y_{i}(\partial/\partial y_{i}))}=\tau_{*}(y_{i}(\partial/\partial y_{i}))+\frac{1}{z}y_{i}\frac{\partial \tau(y)}{\partial y_{i}}\circ_{\tau},
\]
by setting $z=0$, we proved that the Batyrev elements satisfy the multiplicative relation. 

It is straightforward from the definition that the Batyrev elements satisfy the linear relation.
\end{proof}

Consider the $I$-function for the bundle $\mathcal{E}_{j}$ 
associated to the toric divisor $D^{S}_{j}$, for $1\leq j\leq m$.
\begin{equation*}
 I_{\mathcal{E}_{j}}(y,z)=e^{\sum\limits_{i=0}^{r}p_{i}logy_{i}/z}\sum\limits_{\beta\in \mathbb{K}_{\mathcal{E}_{j}}}\prod\limits_{i=1}^{m+l+2}\left(\frac{\prod_{k=\lceil \langle \hat{D}^{S}_{i},\beta\rangle \rceil}^{\infty}\left(\hat{D}_{i}+\left(\langle \hat{D}^{S}_{i},\beta\rangle-k\right)z\right)}{\prod_{k=0}^{\infty}\left(\hat{D}_{i}+\left(\langle \hat{D}^{S}_{i},\beta\rangle-k\right)z\right)}\right)y^{\beta}\textbf{1}_{v(\beta)},
\end{equation*}
where  $y^{\beta}=y_{0}^{\langle p_{0}^{S},\beta,\rangle}y_{1}^{\langle p_{1}^{S},\beta\rangle}\cdots y_{r+l}^{\langle p_{r+l}^{S},\beta\rangle}$. The following lemma is a generalization of lemma 3.11 in \cite{GI}.

\begin{lemma}\label{lemma:pde}
The $I$-function $I_{\mathcal{E}_{j}}$ 
of the bundle $\mathcal{E}_{j}$, associated to the toric divisor $D^{S}_{j}$, 
satisfies the following partial differential equation:
\begin{equation}\label{I-function-pde}
z\frac{\partial}{\partial y_{0}}\left(y_{0}\frac{\partial}{\partial y_{0}}\right)I_{\mathcal{E}_{j}}=\left(\sum\limits^{r+l}_{i=1}m_{ij}\left(y_{i}\frac{\partial}{\partial y_{i}}\right)-y_{0}\frac{\partial}{\partial y_{0}}\right)I_{\mathcal{E}_{j}}
\end{equation}
\end{lemma}
\begin{proof}
Consider the left hand side of the equation (\ref{I-function-pde}), 
\begin{align*}
& z\frac{\partial}{\partial y_{0}}\left(y_{0}\frac{\partial}{\partial y_{0}}\right)I_{\mathcal{E}_{j}}\\
&=e^{\sum\limits_{i=0}^{r}p_{i}logy_{i}/z}\sum\limits_{\beta\in \mathbb{K}_{\mathcal{E}_{j}}}\prod\limits_{i=1}^{m+l+2}\left(\frac{\prod_{k=\lceil \langle \hat{D}^{S}_{i},\beta\rangle \rceil}^{\infty}\left(\hat{D}_{i}+\left(\langle \hat{D}^{S}_{i},\beta\rangle-k\right)z\right)}{\prod_{k=0}^{\infty}\left(\hat{D}_{i}+\left(\langle \hat{D}^{S}_{i},\beta\rangle-k\right)z\right)}\right)\left(2p_{0}\langle p^{S}_{0},\beta\rangle+\langle p_{0}^{S},\beta\rangle^{2}z\right)\left(y^{\beta}/y_{0}\right)\textbf{1}_{v(\beta)},
\end{align*}
and the right hand side of the equation (\ref{I-function-pde})
\begin{align*}
&\left(\sum\limits^{r+l}_{i=1}m_{ij}\left(y_{i}\frac{\partial}{\partial y_{i}}\right)-y_{0}\frac{\partial}{\partial y_{0}}\right)I_{\mathcal{E}_{j}}\\
& = e^{\sum\limits_{i=0}^{r}p_{i}logy_{i}/z}\sum\limits_{\beta\in \mathbb{K}_{\mathcal{E}_{j}}}\prod\limits_{i=1}^{m+l+2}\left(\frac{\prod_{k=\lceil \langle \hat{D}^{S}_{i},\beta\rangle \rceil}^{\infty}
\left(\hat{D}_{i}+\left(\langle \hat{D}^{S}_{i},\beta\rangle-k\right)z\right)}
{\prod_{k=0}^{\infty}
\left(\hat{D}_{i}+\left(\langle \hat{D}^{S}_{i},\beta\rangle-k\right)z\right)}\right)
\left(\hat{D}_{j}/z+\langle \hat{D}^{S}_{j},\beta\rangle\right)y^{\beta}\textbf{1}_{v(\beta)}.
\end{align*}
It is suffice to prove the coefficients of $y^{\beta}\textbf{1}_{v(\beta)}$ in them are the same, for all $\beta\in \mathbb{K}_{\mathcal{E}_{j}}$. Note that, we can rewrite the product factor 
\begin{equation*}
\frac{\prod_{k=\lceil \langle \hat{D}^{S}_{i},\beta\rangle \rceil}^{\infty}
\left(\hat{D}_{i}+\left(\langle \hat{D}^{S}_{i},\beta\rangle-k\right)z\right)}
{\prod_{k=0}^{\infty}\left(\hat{D}_{i}+\left(\langle \hat{D}^{S}_{i},\beta\rangle-k\right)z\right)}
=\frac{\prod_{k\leq 0, \{k\}=\{\langle \hat{D}^{S}_{i},\beta\rangle\}}\left(\hat{D}_{i}+kz\right)}{\prod_{k\leq \langle\hat{D}^{S}_{i},\beta\rangle, \{k\}=\{\langle \hat{D}^{S}_{i},\beta\rangle\}}\left(\hat{D}_{i}+kz\right)}.
\end{equation*}
Let $\beta^{'}=\beta+[\sigma_{0}]$, hence we have
\[
\langle \hat{D}^{S}_{j},\beta^{'}\rangle =\langle \hat{D}^{S}_{j},\beta\rangle -1; \quad \langle \hat{D}^{S}_{i},\beta^{'}\rangle= \langle \hat{D}^{S}_{i},\beta\rangle\text{ for }1\leq i\leq m+l\text{ and }i\neq j;
\] 
\[
\langle \hat{D}^{S}_{m+l+1},\beta^{'}\rangle=\langle \hat{D}^{S}_{m+l+1},\beta\rangle+1; \quad \langle \hat{D}^{S}_{m+l+2},\beta^{'}\rangle=\langle \hat{D}^{S}_{m+l+2},\beta\rangle+1.
\]
Note that $\beta \in \mathbb{K}_{\mathcal{E}_{j}}$ if and only if $\beta^{'}\in \mathbb{K}_{\mathcal{E}_{j}}$. Moreover,
\[
\left(y^{\beta^{'}}/y_{0}\right)\textbf{1}_{v(\beta^{'})}=y^{\beta}\textbf{1}_{v(\beta)}.
\] 
Hence the coefficient of $y^{\beta}\textbf{1}_{v(\beta)}$ in $z\frac{\partial}{\partial y_{0}}(y_{0}\frac{\partial}{\partial y_{0}})I_{\mathcal{E}_{j}}$ is
\begin{align*}
& e^{\sum\limits_{i=0}^{r}p_{i}logy_{i}/z}
\prod\limits_{i=1}^{m+l+2}
\left(\frac{\prod_{k=\lceil \langle \hat{D}^{S}_{i},\beta\rangle \rceil}^{\infty}\left(\hat{D}_{i}+\left(\langle \hat{D}^{S}_{i},\beta\rangle-k\right)z\right)}{\prod_{k=0}^{\infty}\left(\hat{D}_{i}+\left(\langle \hat{D}^{S}_{i},\beta\rangle-k\right)z\right)}\right)\frac{\hat{D}_{j}+\langle \hat{D}^{S}_{j},\beta\rangle z}{\left(p_{0}+(\langle p^{S}_{0},\beta\rangle+1)z\right)^{2}}\bullet\\
&  \bullet(2p_{0}(\langle p^{S}_{0},\beta\rangle+1)+(\langle p_{0}^{S},\beta\rangle+1)^{2}z)\\
=& e^{\sum\limits_{i=0}^{r}p_{i}logy_{i}/z}
\prod\limits_{i=1}^{m+l+2}
\left(\frac{\prod_{k=\lceil \langle \hat{D}^{S}_{i},\beta\rangle \rceil}^{\infty}
\left(\hat{D}_{i}+\left(\langle \hat{D}^{S}_{i},\beta\rangle-k\right)z\right)}
{\prod_{k=0}^{\infty}
\left(\hat{D}_{i}+\left(\langle \hat{D}^{S}_{i},\beta\rangle-k\right)z\right)}
\right)\frac{\hat{D}_{j}+\langle \hat{D}^{S}_{j},\beta\rangle z}{z}\qquad  (\text{since } p_{0}^{2}=0).
\end{align*}
This is exactly the coefficient of $y^{\beta}\textbf{1}_{v(\beta)}$ in $\left(\sum\limits^{r+l}_{i=1}m_{ij}\left(y_{i}\frac{\partial}{\partial y_{i}}\right)-y_{0}\frac{\partial}{\partial y_{0}}\right)I_{\mathcal{E}_{j}}$,

Hence the I-function of $\mathcal{E}_{j}$ satisfies the differential equation
\[
z\frac{\partial}{\partial y_{0}}\left(y_{0}\frac{\partial}{\partial y_{0}}\right)I_{\mathcal{E}_{j}}=\left(\sum\limits^{r+l}_{i=1}m_{ij}\left(y_{i}\frac{\partial}{\partial y_{i}}\right)-y_{0}\frac{\partial}{\partial y_{0}}\right)I_{\mathcal{E}_{j}}.
\]
\end{proof}

Using the expansion of $I_{\mathcal{E}_{j}}$, we have
\begin{equation*}
I_{\mathcal{E}_{j}}(y,z)=e^{\sum\limits_{i=0}^{r}p_{i}logy_{i}/z}
\left(1+z^{-1}\left(\sum\limits_{i=0}^{r}g_{i}^{(j)}(y)p_{i}+\tau^{(j)}_{tw}\right)
+z^{-2}\left(\sum\limits_{n=0}^{2}G_{n}^{(j)}(y)y_{0}^{n}\right)+O(z^{-3})\right),
\end{equation*}
where $G_{n}^{(j)}$ is a (fractional) power series in $y_{1},\ldots,y_{r+l}$ taking values in $H^{*}_{orb}(\mathcal{E}_{j})$.
Therefore, we obtain
\begin{align*}
y_{0}\frac{\partial}{\partial y_{0}}I_{\mathcal{E}_{j}}=& \frac{p_{0}}{z}e^{\sum\limits_{i=0}^{r}p_{i}logy_{i}/z}\left(1+z^{-1}\left(\sum\limits_{i=0}^{r}g_{i}^{(j)}(y)p_{i}+\tau^{(j)}_{tw}\right)+z^{-2}\left(\sum\limits_{n=0}^{2}G_{n}^{(j)}(y)y_{0}^{n}\right)+O(z^{-3})\right)\\
& +e^{\sum\limits_{i=0}^{r}p_{i}logy_{i}/z}\left(z^{-2}\left(\sum\limits_{n=1}^{2}G_{n}^{(j)}(y)ny_{0}^{n}\right)+O(z^{-3})\right).
\end{align*}
Therefore, the left hand side of equation (\ref{I-function-pde}) is
\begin{align*}
& z\frac{\partial}{\partial y_{0}}\left(y_{0}\frac{\partial}{\partial y_{0}}\right)I_{\mathcal{E}_{j}}\\
= & \frac{\partial}{\partial y_{0}}\left(p_{0}e^{\sum\limits_{i=0}^{r}p_{i}logy_{i}/z}\left(1+z^{-1}\left(\sum\limits_{i=0}^{r}g_{i}^{(j)}(y)p_{i}+\tau^{(j)}_{tw}\right)+z^{-2}\left(\sum\limits_{n=0}^{2}G_{n}^{(j)}(y)y_{0}^{n}\right)+O(z^{-3})\right)\right)\\
&+\frac{\partial}{\partial y_{0}}\left(e^{\sum\limits_{i=0}^{r}p_{i}logy_{i}/z}\left(z^{-1}\left(\sum\limits_{n=1}^{2}G_{n}^{(j)}(y)ny_{0}^{n}\right)+O(z^{-2})\right)\right)\\
=& p_{0}e^{\sum\limits_{i=0}^{r}p_{i}logy_{i}/z}\left(O(z^{-2})\right) +\frac{p_{0}}{y_{0}z}e^{\sum\limits_{i=0}^{r}p_{i}logy_{i}/z}\left(z^{-1}\left(\sum\limits_{n=1}^{2}G_{n}^{(j)}(y)ny_{0}^{n}\right) +O(z^{-2})\right) \\
& + e^{\sum\limits_{i=0}^{r}p_{i}logy_{i}/z}\left(z^{-1}\left(\sum\limits_{n=1}^{2}G_{n}^{(j)}n^{2}y_{0}^{n-1} +O(z^{-2})\right)\right)\\
=&e^{\sum\limits_{i=0}^{r}p_{i}logy_{i}/z}\left(z^{-1}\left(\sum\limits_{n=1}^{2}G_{n}^{(j)}n^{2}y_{0}^{n-1}\right)+O(z^{-2})\right).\\
\end{align*}
On the other hand, the pull-back of the right hand side of equation (\ref{I-function-pde}) by $\imath^{*}$ is 
\begin{align*}
& \imath^{*}\left(\sum\limits^{r+l}_{i=1}m_{ij}\left(y_{i}\frac{\partial}{\partial y_{i}}\right)-y_{0}\frac{\partial}{\partial y_{0}}\right)I_{\mathcal{E}_{j}}\\
=& \left(\sum\limits^{r+l}_{i=1}m_{ij}\left(y_{i}\frac{\partial}{\partial y_{i}}\right)-y_{0}\frac{\partial}{\partial y_{0}}\right)\imath^{*}I_{\mathcal{E}_{j}}\\
=& \left(\sum\limits^{r+l}_{i=1}m_{ij}\left(y_{i}\frac{\partial}{\partial y_{i}}\right)\right)\left(I_{\mathcal{X}}+O(y_{0})\right)\\
=& z^{-1}\left(\sum\limits_{i=1}^{r+l}m_{ij}\left(y_{i}\frac{\partial}{\partial y_{i}}\right)\tau(y)\right)+O(z^{-2})+O(y_{0}).
\end{align*}
Hence we conclude the following lemma.
\begin{lemma} 
The Batyrev element $\tilde{D}_{j}(y)$ is given by 
\begin{equation}
\tilde{D}_{j}(y)=\imath^{*}G_{1}^{(j)}(y),\quad \text{for} \quad 1\leq j\leq m+l.
\end{equation} 
\end{lemma}
Hence, the following theorem is a direct consequence of the above lemma and theorem \ref{Seidel-I-function}.
\begin{theorem}
The Seidel element $\tilde{S}_{j}$ corresponding to the toric divisor $D_{j}$ is given by  
\begin{equation}
\tilde{S}_{j}(\tau(y))=exp(-g_{0}^{j}(y))\tilde{D}_{j}(y).
\end{equation}
\end{theorem}

\subsection{The computation of $\tilde{D}_{j}$}

Using the expansion 
\[
\left(\sum\limits^{r+l}_{i=1}m_{ij}\left(y_{i}\frac{\partial}{\partial y_{i}}\right)\right)I_{\mathcal{X}}=e^{\sum_{i=1}^{r}p_{i}logy_{i}/z}\left(z^{-1}\tilde{D}_{j}+O(z^{-2})\right),
\] 
we see that $\tilde{D}_{j}$
is the coefficient of $z^{-1}$ in the expansion of  
\[
e^{-\sum_{i=1}^{r}p_{i}logy_{i}/z}\left(\sum\limits^{r+l}_{i=1}m_{ij}\left(y_{i}\frac{\partial}{\partial y_{i}}\right)\right)I_{\mathcal{X}}.
\]
And, by direct computation
\begin{align*}
& \left(\sum\limits^{r+l}_{i=1}m_{ij}\left(y_{i}\frac{\partial}{\partial y_{i}}\right)\right)I_{\mathcal{X}}=\\
& e^{\sum\limits_{i=1}^{r}p_{i}logy_{i}/z}\sum\limits_{d\in \mathbb{K}_{\text{eff},\mathcal{X}}}\prod\limits_{i=1}^{m+l}\left(\frac{\prod_{k=\lceil \langle D^{S}_{i},d\rangle \rceil}^{\infty}\left({D}_{i}+\left(\langle D^{S}_{i},d\rangle-k\right)z\right)}{\prod_{k=0}^{\infty}\left({D}_{i}+\left(\langle D^{S}_{i},d\rangle-k\right)z\right)}\right)\left(\frac{D_{j}}{z}+\langle D^{S}_{j},d\rangle\right)y^{d}\textbf{1}_{v(d)}.\\
\end{align*}
Hence, to compute the Batyrev element $\tilde{D}_{j}$, it remains to examine the expansion of the product factor
\begin{equation*}
\frac{\prod_{k=\lceil \langle D^{S}_{i},d\rangle \rceil}^{\infty}\left({D}_{i}+\left(\langle D^{S}_{i},d\rangle-k\right)z\right)}{\prod_{k=0}^{\infty}\left({D}_{i}+\left(\langle D^{S}_{i},d\rangle-k\right)z\right)}
= C_{d}z^{-\left(\sum_{i=1}^{m+l}\lceil \langle D^{S}_{i},d\rangle \rceil +\# \{i: \langle D_{i}^{S},d\rangle \in \mathbb{Z}_{<0}\}\right)}\prod_{i:\langle D^{S}_{i},d\rangle \in \mathbb{Z}_{<0}}D_{i}+h.o.t.,
\end{equation*}
where 
\begin{equation}\label{C_d}
C_{d}=\prod\limits_{i:\langle D^{S}_{i},d\rangle <0}\prod\limits_{\langle D^{S}_{i},d\rangle <k<0}\left(\langle D^{S}_{i},d\rangle -k\right)\prod\limits_{i:\langle D^{S}_{i},d\rangle >0}\prod\limits_{0\leq k<\langle D^{S}_{i},d\rangle}\left(\langle D^{S}_{i},d\rangle -k\right)^{-1}
\end{equation}

The summand indexed by $d\in \mathbb{K}_{\text{eff},\mathcal{X}}$ contributes to the coefficient of $z^{-1}$ if and only if
\[ 
\sum_{i=1}^{m+l}\lceil \langle D^{S}_{i},d\rangle \rceil +\# \{i: \langle D_{i}^{S},d\rangle \in \mathbb{Z}_{<0}\}\leq 1.
\]
It happens only in the following three cases:
\begin{itemize}
\item $\sum_{i=1}^{m+l}\lceil \langle D^{S}_{i},d\rangle \rceil +\# \{i: \langle D_{i}^{S},d\rangle \in \mathbb{Z}_{<0}\}=0$

\item 
 $\left\{
     \begin{array}{lr}
     \sum_{i=1}^{m+l}\lceil \langle D^{S}_{i},d\rangle \rceil= 0 \\
      \# \{i: \langle D_{i}^{S},d\rangle \in \mathbb{Z}_{<0}\}=1
     \end{array}
   \right. $
\item
$\left\{
     \begin{array}{lr}
     \sum_{i=1}^{m+l}\lceil \langle D^{S}_{i},d\rangle \rceil= 1 \\
      \# \{i: \langle D_{i}^{S},d\rangle \in \mathbb{Z}_{<0}\}=0
     \end{array}
   \right. $.
\end{itemize}
The first case happens if and only if $d=0$.
If the second case happens, then 
\[
\sum_{i=1}^{m+l}\lceil \langle D^{S}_{i},d\rangle \rceil=  \sum_{i=1}^{m+l}\langle D^{S}_{i},d\rangle=\langle\rho^{S}_{\mathcal{X}},d\rangle=0.
\]
In particular, 
\[
\langle D^{S}_{i},d\rangle \in\mathbb{Z}, 1\leq i \leq m+l.
\] 
Hence we obtain the following lemma:
\begin{lemma}
For $1\leq j \leq m+l$, the Batyrev element $\tilde{D}_{j}$ is given by 
\begin{equation}
\tilde{D}_{j}=D_{j}+\sum\limits_{i=1}^{m}D_{i}\sum\limits_{\substack{\langle\rho^{S}_{\mathcal{X}},d\rangle=0\\\langle D^{S}_{i},d\rangle \in \mathbb{Z}_{<0}\\ \langle D_{k}^{S},d\rangle \in \mathbb{Z}_{\geq 0}, \forall k\neq i}}C_{d}\langle D^{S}_{j},d\rangle y^{d}+\sum\limits_{\substack{ \sum_{i=1}^{m+l}\lceil \langle D^{S}_{i},d\rangle \rceil= 1\\\langle D^{S}_{i},d\rangle \not\in \mathbb{Z}_{<0}, \forall i}}C_{d}\langle D^{S}_{j},d\rangle y^{d}\textbf{1}_{v(d)},
\end{equation}
where $C_{d}$ is given by equation (\ref{C_d}).
\end{lemma}

\section{Seidel elements corresponding to Box elements}
Consider the box element $s_{j}\in Box(\boldsymbol{\Sigma})$, 
such that 
\[
\bar{s}_{j}=\sum\limits_{i=1}^{m}c_{ji}\bar{b}_{i}\in \textbf{N}_{\mathbb{Q}},\quad \text{for some}\quad 0\leq c_{ji}<1.
\] 
Let $n_{j}$ be the least common denominator of 
$\{c_{ji}\}_{i=1}^{m}$, we define a 
$\mathbb{C}^{\times}$-action on $\mathcal{U}^{S}\times (\mathbb{C}^{2}\setminus \{0\})$ by 
\[
(z_{1},\ldots,z_{m+l},u,v)\mapsto (t^{-c_{j1}n_{j}}z_{1},\ldots,t^{-c_{jm}n_{j}}z_{m},z_{m+1},\ldots,z_{m+l},t^{n_{j}}u,t^{n_{j}}v), \quad t\in \mathbb{C}^{\times}.
\]
Hence we have an associated bundle 
\[
\mathcal{E}_{m+j}=\mathcal{U}^{S}\times (\mathbb{C}^{2}\setminus \{0\})/G^{S}\times \mathbb{C}^{\times}
\] 
over $\mathbb{CP}^{1}\times B\mu_{n_{j}} $ 
with $\mathcal{X}$ being the fiber. Furthermore, $\mathcal{E}_{m+j}$ can also be considered as a bundle over $\mathbb{CP}^{1}$, since there is a natural projection 
\[
\mathbb{CP}^{1}\times B\mu_{n_{j}}\rightarrow \mathbb{CP}^{1}.
\]

We can identify $H^{2}(\mathcal{E}_{m+j};\mathbb{Z})$ with $H^{2}(\mathcal{X};\mathbb{Z})\oplus \mathbb{Z}$, where the second summand 
\[
\mathbb{Z}\cong Pic(\mathbb{CP}^{1}\times B\mu_{n_{j}}),
\] 
and we have the following short exact sequence from remark 5.5 of \cite{FMN}:
\begin{equation}
0\longrightarrow Pic(\mathbb{CP}^{1})\longrightarrow Pic(\mathbb{CP}^{1}\times B\mu_{n_{j}})\longrightarrow \mathbb{Z}/n_{j}\mathbb{Z}\longrightarrow 0
\end{equation}
We identify an element of $Pic(\mathbb{CP}^{1})$ with its image in $Pic(\mathbb{CP}^{1}\times B\mu_{n_{j}})$ under the above map. Then the weights of $G^{S}\times \mathbb{C}^{\times}$ defining $\mathcal{E}_{m+j}$ are given by

\[
\hat{D}^{S}_{i}=(D^{S}_{i},-c_{ji}n_{j}), \quad \text{for}\quad 1\leq i \leq m; \quad \hat{D}^{S}_{m+j}=(D^{S}_{m+j},0)\quad \text{for}\quad 1\leq j \leq l;
\] 
\[
\hat{D}^{S}_{m+l+1}=\hat{D}^{S}_{m+l+2}=(0,n_{j}).
\]

The fan of $\mathcal{E}_{m+j}$ is contained in 
$N_{\mathbb{Q}}\oplus \mathbb{Q}$. 
The 1-skeleton is given by 
\begin{equation}
\hat{b}_{i}=(b_{i},0),\text{ for }1\leq i \leq m;\quad \hat{b}_{m+1}=(0,1);\quad \hat{b}_{m+2}=(s_{j},-1).
\end{equation}
Let $E_{m+j}$ be the coarse moduli space of 
$\mathcal{E}_{m+j}$. 
Then $E_{m+j}$ is an $X$-bundle over 
$\mathbb{CP}^{1}$. 
The Seidel element is defined as in equation (\ref{seidelelement}).

We set 
\[
p_{0}:=(0,1)\in H^{2}(E_{m+j})\cong H^{2}(X)\oplus Pic(\mathbb{CP}^{1}),
\] 
a nef integral basis $\{p_{1},\ldots,p_{r}\}$ of $H^{2}(X;\mathbb{Q})$ 
can be lifted to a nef integral basis 
$\{p_{0},p_{1},\ldots,p_{r}\}$ of 
$H^{2}(E_{m+j};\mathbb{Q})$ such that the lift of 
$p_{i}$ vanishes on the section class $[\sigma_{0}]$. 
There is an isomorphism between 
$H^{2}(E_{m+j};\mathbb{Q})$ and $H^{2}(\mathcal{E}_{m+j};\mathbb{Q})$, 
by abuse of notation, we identify $p_{i}$ with its image in 
$H^{2}(\mathcal{E}_{m+j};\mathbb{Q})$,
 for $0\leq i \leq r$.  
Let $p_{1}^{S},\ldots,p_{r+l}^{S}$ 
be an integral basis of 
$\mathbb{L}^{S\vee}$, 
such that $p_{i}$ is the image of 
$p_{i}^{S}$ in $\mathbb{L}^{\vee}\otimes\mathbb{Q}$, under the canonical splitting of (\ref{L^S-splitting}).  
Let $p_{0}^{S}, p_{1}^{S},\ldots,p_{r+l}^{S}$ 
be an integral basis of $\mathbb{L}^{S\vee}\oplus \mathbb{Z}$ and $p_{0}$ be the image of
\[ 
p_{0}^{S}=\hat{D}_{m+l+1}^{S}=\hat{D}_{m+l+2}^{S}
\]
in $(\mathbb{L}^{\vee}\oplus\mathbb{Z})\otimes \mathbb{R}$. Therefore $p_{r+1},\ldots,p_{r+l}$ are zero.

As in the toric divisor case, we have the following expansion of the $I$-function:
\begin{align}
&  I_{\mathcal{E}_{m+j}}(y,z)=\\
&\notag e^{\sum\limits_{i=0}^{r}p_{i}logy_{i}/z}
\left(1+z^{-1}\left(\sum\limits_{i=0}^{r}g_{i}^{(m+j)}(y)p_{i}+\tau^{(m+j)}_{tw}(y)\right)+z^{-2}\left(\sum\limits_{n=0}^{2}G_{n}^{(m+j)}(y)y_{0}^{n}\right)+O(z^{-3})\right),
\end{align}
and use the same argument as in lemma \ref{g^(j)-indep-y_0} and lemma \ref{tau-indep-y_0}, we can show that $g_{i}^{(m+j)}(y)$ and $\tau^{(m+j)}_{tw}(y)$ are independent from $y_{0}$, for $1\leq i \leq r$ and $1\leq j \leq l$ . Moreover, for each $j\in\{1,\ldots,l\}$, we have 
\[
g_{i}^{(m+j)}(y_{0},\ldots,y_{r+l})=g_{i}(y_{1},\ldots,y_{r+l})\quad  \text{for} \quad i=1,\ldots,r.
\]
And 
\[
\tau_{tw}^{(m+j)}(y)=\tau_{tw}(y).
\]
We will also obtain the following theorem.
\begin{thm}
The Seidel element $\tilde{S}_{m+j}$ associated to the box element $s_{j}$ is given by 
\begin{equation}
\tilde{S}_{m+j}(\tau(y)):=\tilde{S}_{m+j}(\tau^{(m+j)}(y))=exp\left(-g_{0}^{(m+j)}(y)\right)\imath^{*}(G_{1}^{(m+j)}(y)).
\end{equation}
\end{thm}
Using the same computation as in the toric divisor case, we can compute the correction coefficient $g_{0}^{(m+j)}$:

\begin{lemma}
The function $g_{0}^{(m+j)}$ is given by 
\begin{equation}\label{correction coefficient 2}
g_{0}^{(m+j)}(y_{1},\ldots,y_{r+l})=\sum\limits_{1\leq k\leq m, k\not\in I_{j}^{S}.}\sum\limits_{\substack{\langle D^{S}_{i},d\rangle\in \mathbb{Z}, 1 \leq i\leq m+l\\ \langle \rho^{S}_{\mathcal{X}},d\rangle=0\\ \langle D_{k}^{S},d\rangle <0\\ \langle D^{S}_{i},d\rangle \geq 0, \forall i \neq k}} c_{jk}\frac{(-1)^{-\langle D_{k}^{S},d\rangle}\left(-\langle D_{k}^{S},d\rangle-1\right)!}{\prod_{i\neq k}\langle D_{i}^{S},d\rangle !}y^{d},
\end{equation}
where $I_{j}^{S}$ is the "anticone" of the cone containing $s_{j}$.
\end{lemma}

\begin{proof}
The argument is almost the same as the argument in section \ref{computation-of-g}. 
The only change we need to make is 
the paragraph above lemma \ref{expression-of-g}:
 
In this case, $g_{0}^{(m+j)}$ 
is the coefficient corresponding to 
$p_{0}$ and elements in $\{\hat{D}_{1},\ldots,\hat{D}_{m}\}$ that contain $p_{0}$ are precisely these elements: 
\[
\hat{D}_{k}=\langle D_{k},-c_{jk}n_{j}\rangle=D_{k}-c_{jk}p_{0}, \quad
\text{for} \quad 1\leq k \leq m\quad \text{and}\quad k\not\in I^{S}_{j}.
\] 
Therefore, by expression (\ref{expansion-C^beta}) and (\ref{third-case}), we must have $\langle D^{S}_{k},d\rangle<0$ for exactly one $k$ in $\{k\in \mathbb{Z}| 1\leq k \leq m$ and $k\not\in I^{S}_{j}\}$. 
\end{proof}

Moreover, by mimicking the computation in lemma \ref{lemma:pde}, we have

\begin{lemma}
the $I$-function of $\mathcal{E}_{m+j}$ satisfies the following differential equation:
\begin{equation}
z\frac{\partial}{\partial y_{0}}\left(y_{0}\frac{\partial}{\partial y_{0}}\right)I_{\mathcal{E}_{j}}=y^{-D_{m+j}^{S\vee}}\left(\sum\limits^{r+l}_{i=1}m_{ij}\left(y_{i}\frac{\partial}{\partial y_{i}}\right)-y_{0}\frac{\partial}{\partial y_{0}}\right)I_{\mathcal{E}_{j}},
\end{equation}
where $D_{m+j}^{S\vee}\in \mathbb{L}^{S}\otimes \mathbb{Q}$ is defined by (\ref{D-dual}).
\end{lemma}
\begin{proof}
The proof is almost identical to the proof of lemma \ref{lemma:pde}, except, this time, we will need to choose $\beta^{'}=\beta+[\sigma_{0}]-D^{S\vee}_{m+j}$. Then everything else follows. 
\end{proof}
Using this lemma, following the argument in the toric divisor case, we conclude
\begin{theorem}
The Seidel element $\tilde{S}_{m+j}$ 
corresponding to the box element 
$s_{j}$, with 
\[
\bar{s}_{j}=\sum\limits_{i=1}^{m}c_{ji}\bar{b}_{i},\quad \text{for some} \quad 0\leq c_{ji}<1,
\]  
is given by
\begin{equation}
\tilde{S}_{m+j}(\tau^{(m+j)}(y))=exp\left(-g_{0}^{(m+j)}\right)y^{-D_{m+j}^{S\vee}}\tilde{D}_{m+j}(y),
\end{equation}
where $\tilde{D}_{m+j}(y)$ is the corresponding Batyrev element. Moreover, 
\begin{equation}
\tilde{D}_{m+j}=\sum\limits_{i=1}^{m}D_{i}\sum\limits_{\substack{\langle\rho^{S}_{\mathcal{X}},d\rangle=0\\\langle D^{S}_{i},d\rangle \in \mathbb{Z}_{<0}\\ \langle D_{k}^{S},d\rangle \in \mathbb{Z}_{\geq 0}, \forall k\neq i}}C_{d}\langle D^{S}_{m+j},d\rangle y^{d}+\sum\limits_{\substack{ \sum_{i=1}^{m+l}\lceil \langle D^{S}_{i},d\rangle \rceil= 1\\\langle D^{S}_{i},d\rangle \not\in \mathbb{Z}_{<0}, \forall i}}C_{d}\langle D^{S}_{m+j},d\rangle y^{d}\textbf{1}_{v(d)},
\end{equation}
and 
\begin{equation}
C_{d}=\prod\limits_{i:\langle D^{S}_{i},d\rangle <0}\prod\limits_{\langle D^{S}_{i},d\rangle <k<0}\left(\langle D^{S}_{i},d\rangle -k\right)\prod\limits_{i:\langle D^{S}_{i},d\rangle >0}\prod\limits_{0\leq k<\langle D^{S}_{i},d\rangle}\left(\langle D^{S}_{i},d\rangle -k\right)^{-1}.
\end{equation}
\end{theorem}

\end{document}